\documentclass[11pt, twoside, leqno]{article}

\usepackage{amssymb}
\usepackage{amsmath}
\usepackage{amsthm}

\usepackage[active]{srcltx}

\allowdisplaybreaks

\def\rr{{\mathbb R}}
\def\rn{{{\rr}^n}}

\def\zz{{\mathbb Z}}
\def\kzz{{k\in\zz_+}}
\def\nn{{\mathbb N}}

\def\cc{{\mathbb C}}

\def\cg{{\mathcal G}}

\def\cx{{\mathcal X}}

\def\az{\alpha}
\def\supp{{\mathop\mathrm{\,supp\,}}}

\def\diam{{\mathop\mathrm{\,diam\,}}}
\def\loc{{\mathop\mathrm{\,loc\,}}}

\def\stz{\stackrel}
\def\noz{{\nonumber}}

\def\fz{\infty}
\def\lz{\lambda}
\def\dz{\delta}

\def\wz{\widetilde}
\def\ez{\epsilon}
\def\wez{\wz\ez}
\def\ezp{{\ez'}}

\def\wez{{\wz\ez}}

\def\bz{\beta}

\def\kz{{\kappa}}
\def\gz{{\gamma}}

\def\vz{\varphi}
\def\tz{\theta}

\def\ls{\lesssim}
\def\gs{\gtrsim}
\def\laz{\langle}
\def\raz{\rangle}
\def\oz{\overline}

\def\wg{\wedge}

\def\lf{\left}
\def\r{\right}

\def\ocg{{\mathring{\cg}}}
\def\cgbz{{\cg(\bz,\gz)}}

\def\cgm{{\cg(x_1,r,\bz,\gz)}}
\def\cgom{{\ocg(x_1,r,\bz,\gz)}}

\def\lp{{L^p(\cx)}}

\def\db{{\dot B^s_{p,\,q}(\cx)}}
\def\df{{\dot F^s_{p,\,q}(\cx)}}

\def\sh{{h^p(\cx)}}

\def\sbmo{{\mathop\mathrm {bmo\,}}}

\def\bint{{\ifinner\rlap{\bf\kern.35em--}
\int\else\rlap{\bf\kern.45em--}\int\fi}\ignorespaces}

\def\bbint{{\ifinner\rlap{\bf\kern.35em--}
\hspace{0.078cm}\int\else\rlap{\bf\kern.45em--}\int\fi}\ignorespaces}

\def\b{{B^s_{p,\,q}(\cx)}}
\def\f{{F^s_{p,\,q}(\cx)}}

\def\dfy{{\dot F^s_{\infty,\, q}(\cx)}}
\def\fy{{F^s_{\infty,\, q}(\cx)}}

\def\nj2{{N_{J_2}}}

\def\qji2{{Q^{J_2}_{i_2}}}

\def\qal{{Q^l_\az}}

\def\qtn{{Q_\tau^{k,\nu}}}
\def\qto{{Q_\tau^{0,\nu}}}
\def\qtop{{Q_{\tau'}^{0,\nu'}}}
\def\ytn{y_\tau^{k,\nu}}
\def\yto{y_\tau^{0,\nu}}

\def\ytop{y_{\tau'}^{0,\nu'}}
\def\nkt{{N(k,\tau)}}
\def\ik{{I_k}}
\def\qtnp{{Q_{\tau'}^{k',\nu'}}}
\def\ytnp{y_{\tau'}^{k',\nu'}}
\def\nktp{{N(k',\tau')}}
\def\ikp{{I_{k'}}}

\def\wdk{{\wz D_k}}
\def\wdkp{{\wz D_{k'}}}
\def\dkp{{D_{k'}}}
\def\wdo{{\wz D_0}}

\def\io{{I_0}}
\def\ik{{I_k}}
\def\ikp{{I_{k'}}}

\def\nz{{N(0,\tau)}}
\def\nzp{{N(0,\tau')}}

\def\hs{\hspace{0.3cm}}

\def\dsum{\displaystyle\sum}
\def\dint{\displaystyle\int}
\def\dfrac{\displaystyle\frac}
\def\dsup{\displaystyle\sup}

\newtheorem{thm}{Theorem}[section]
\newtheorem{lem}{Lemma}[section]
\newtheorem{prop}{Proposition}[section]
\newtheorem{rem}{Remark}[section]
\newtheorem{cor}{Corollary}[section]
\newtheorem{defn}{Definition}[section]

\numberwithin{equation}{section}

\begin{document}

\arraycolsep=1pt

\title{{\vspace{-5cm}\small\hfill\bf Manuscripta Math., to appear}\\
\vspace{4.5cm}\bf New Properties of Besov and Triebel-Lizorkin Spaces
on RD-Spaces\footnotetext {\hspace{-0.35cm}
2000 {\it Mathematics Subject Classification}.
Primary 46E35; Secondary 42B30, 43A99.\endgraf {\it
Key words and phrases}. RD-space, Besov space, Triebel-Lizorkin
space, local Hardy space, Calder\'on reproducing formula, space of test functions,
local integrability.
\endgraf The first author is supported by the National
Natural Science Foundation (Grant No. 10871025) of China.}}
\author{Dachun Yang and Yuan Zhou}
\date{ }
\maketitle

\begin{center}
\begin{minipage}{13.5cm}
{\small {\bf Abstract} \quad An RD-space $\mathcal X$ is a space of
homogeneous type in the sense of Coifman and Weiss with the
additional property that a reverse doubling property holds in $\mathcal X$.
In this paper, the authors first give several equivalent
characterizations of RD-spaces
and show that the definitions of spaces
of test functions on $\mathcal X$ are independent of the choice of the
regularity $\epsilon\in (0,1)$; as a result of this,
the Besov and Triebel-Lizorkin spaces on $\mathcal X$ are also independent of
the choice of the underlying distribution space.
Then the authors characterize the
norms of inhomogeneous Besov and Triebel-Lizorkin spaces by the
norms of homogeneous Besov and Triebel-Lizorkin spaces together with
the norm of local Hardy spaces in the sense of Goldberg. Also, the
authors obtain the sharp locally integrability of elements in Besov
and Triebel-Lizorkin spaces.}
\end{minipage}
\end{center}

\medskip

\section{Introduction\label{s1}}

\hskip\parindent The theory of Besov and Triebel-Lizorkin spaces plays an important
role in various fields of mathematics such as harmonic analysis,
partial differential equations, geometric analysis and etc; see, for example,
\cite{t06,t83,t92,t97}. Recently, Besov and Triebel-Lizorkin spaces
on metric measure spaces obtained a rapid development; see
\cite{hmy2,my, gks, t06}.

We first recall the definitions of spaces of homogenous type
in \cite{cw1,cw2} and RD-spaces in \cite{hmy2}.
In this paper, we always assume that
$(\cx, d)$ is a metric space with a regular Borel measure $\mu$
such that all balls defined by $d$ have finite and positive measures.
In what follows, set
$\diam(\cx)\equiv\sup\{d(x,\,y):\ x,\ y\in\cx\}$
and for any $x\in\cx$ and $r>0$, set $B(x, r)\equiv\{y\in\cx:\, d(x, y)<r\}$.

\begin{defn}\label{d1.1}
(i) The triple $(\cx, d, \mu)$ is called a space of homogeneous type
if there exists a constant $C_0\in(1,\fz)$
such that for all $x\in\cx$ and $r>0$,
\begin{equation}\label{1.1}
\mu(B(x, 2r))\le C_0\mu(B(x, r)) \ (doubling\ property).
\end{equation}

(ii) The triple  $(\cx, d, \mu)$
is called an RD-space if it is a space of homogeneous type
and there exist constants $a_0,\,\wz C_0\in(1,\fz)$
such that for all $x\in\cx$ and $0<r<\diam(\cx)/a_0$,
\begin{equation*}
\mu(B(x, a_0r))\ge \wz C_0\mu(B(x, r)) \ (reverse\ doubling\ property).
\end{equation*}
\end{defn}

It is easy to see that $\cx$ is an RD-space if and only if
it is a space of homogeneous type
and there exists a constant $\wz a_0\in(1,\fz)$
such that for all $x\in\cx$ and $0<r<\diam(\cx)/\wz a_0$,
$\mu(B(x, \wz a_0r))\ge 2\mu(B(x, r))$. We will establish several
other equivalent characterizations of RD-spaces
in Proposition \ref{p2.1} below.

The following spaces of test functions play a key role in the theory
of function spaces on RD-spaces; see \cite{hmy1, hmy2}. In what follows,
for any $x,$ $y\in\cx$ and $r>0$, set $V(x, y)\equiv\mu(B(x, d(x,y)))$
and $V_r(x)\equiv\mu(B(x, r))$.

\begin{defn}\hspace{-0.2cm}{\bf}\label{d1.2}
Let $x_1\in\cx$, $r\in(0, \fz)$, $\bz\in(0, 1]$ and $\gz\in(0,
\fz)$. A function $\vz$ on $\cx$ is called a test function of
type $(x_1, r, \bz, \gz)$ if there exists a nonnegative constant
$C$ such that

(i) $|\vz(x)|\le C\frac{1}{V_r(x_1)+V(x_1, x)}
\lf[\frac{r}{r+d(x_1, x)}\r]^\gz$ for all $x\in\cx$;

(ii) $|\vz(x)-\vz(y)| \le C\lf[\frac{d(x, y)}{r+d(x_1, x)}\r]^\bz
\frac{1}{V_r(x_1)+V(x_1, x)} \lf[\frac{r}{r+d(x_1, x)}\r]^\gz$
for all $x$, $y\in\cx$ satisfying that $d(x, y)\le[r+d(x_1, x)]/2$.

The space $\cg(x_1, r, \bz, \gz)$ of test functions
is defined to be the set of all test functions
of type $(x_1, r, \bz, \gz)$. If $\vz\in\cg(x_1, r, \bz,\gz)$, its
norm is defined by
$\|\vz\|_{\cg(x_1,\, r,\, \bz,\,\gz)}
\equiv\inf\{C:\, (i) \mbox{ and } (ii) \mbox{ hold}\}$.
\end{defn}

Throughout the whole paper, we fix $x_1\in \cx $ and let $\cg(\bz,\gz)
\equiv\cg(x_1,1,\bz,\gz).$ It is easy to see that
$\cg(\bz, \gz)$ is a Banach space.

For any given $\ez\in(0, 1]$, let $\cg_0^\ez(\bz, \gz)$ be the
completion of the space $\cg(\ez, \ez)$ in $\cg(\bz, \gz)$
when $\bz$, $\gz\in(0, \ez]$.
Obviously, $\cg_0^\ez(\ez, \ez)=\cg(\ez, \ez)$.
Moreover, it is easy to see that $\vz\in\cg_0^\ez(\bz, \gz)$ if and only if
$\vz\in\cg(\bz, \gz)$ and there exists
$\{\phi_i\}_{i\in\nn}\subset\cg(\ez, \ez)$ such that
$\|\vz-\phi_i\|_{\cg(\bz, \gz)}\to0$ as $i\to\fz$.
If $\vz\in\cg_0^\ez(\bz, \gz)$,
define $\|\vz\|_{\cg_0^\ez(\bz, \gz)}\equiv\|\vz\|_{\cg(\bz, \gz)}$.
Obviously, $\cg_0^\ez(\bz, \gz)$ is a Banach space and
$\|\vz\|_{\cg_0^\ez(\bz, \gz)}=\lim_{i\to\fz}\|\phi_i\|_{\cg(\bz, \gz)}$
for the above chosen $\{\phi_i\}_{i\in\nn}$.
Let $(\cg_0^\ez(\bz, \gz))'$ be the set of
all bounded linear functionals $f$ from $\cg_0^\ez(\bz, \gz)$ to $\cc$.
Denote by $\laz f, \vz\raz$ the
natural pairing of elements $f\in (\cg_0^\ez(\bz, \gz))'$ and
$\vz\in\cg_0^\ez(\bz, \gz)$.

In what follows, we define
$$\cgom\equiv\lf\{f\in\cgm:\,\int_\cx f(x)\,d\mu(x)=0\r\}.$$
The space $\ocg_0^\ez(\bz,\gz)$
is defined to be the completion
of the space $\ocg(\ez,\ez)$ in $\ocg(\bz,\gz)$ when
$\bz,\ \gz\in(0,\ez]$. Moreover, if $f\in\ocg_0^\ez(\bz,\gz)$,
we then define $\|f\|_{\ocg_0^\ez(\bz,\gz)}=\|f\|_{\cg(\bz,\gz)}$.

One of the main targets of this paper
is to show that spaces  of
test functions are independent of the choices of $\ez\in (0,1)$ via
the continuous Calder\'on reproducing formulae.

\begin{thm}\label{t1.1} Let $\ez,\ \wz\ez\in (0,1)$ and
$0<\bz,\ \gz<(\ez\wg\wez)$. Then
$\cg^\ez_0(\bz,\gz)=\cg^\wez_0(\bz,\gz)$ and
$\ocg^\ez_0(\bz,\gz)=\ocg^\wez_0(\bz,\gz)$.
\end{thm}

The proof of Theorem \ref{t1.1} will be given in Section \ref{s3}.

Based on the above spaces of test functions, the homogeneous
Besov spaces $\db$ and Triebel-Lizorkin
spaces $\df$, and the inhomogeneous Besov spaces $\b$ and Triebel-Lizorkin
spaces $\f$ on RD-spaces were introduced in \cite{hmy2};
see also Definitions \ref{d4.1} through \ref{d4.3} below.
As a corollary of Theorem \ref{t1.1}, we see that
the Besov and Triebel-Lizorkin spaces are independent of
the regularity of the underlying distribution space; see Corollary \ref{c4.1} below.

Recall that the local Hardy space $h^p(\cx)$ with $p\in(n/(n+1),\,\fz)$
is just the inhomogeneous Triebel-Lizorkin space $F^0_{p,\,2}(\cx)$;
see \cite[Theorem 5.42]{hmy2}.

Recently, Koskela and Saksman \cite{ks08}
characterized the Hardy-Sobolev space by using some
Haj\l asz-Sobolev space on $\rn$.
Motivated by this, as another main target of this paper,
in Theorem \ref{t1.2} below, we characterize
the norms of inhomogeneous Besov spaces and Triebel-Lizorkin spaces
by the norms of homogeneous Besov and Triebel-Lizorkin spaces
together with the norm of local Hardy spaces in the sense of
Goldberg (see \cite{g}).

\begin{thm}\label{t1.2}
Let $\mu(\cx)=\fz$ and $\cx$ have the ``dimension" $n$
as in \eqref{2.2} below.
Let $s\in(0,1)$ and $p\in (n/(n+s), \fz)$. Let
$\{S_k\}_{k\in\zz}$ be an approximation of the identity of order $1$
with bounded support as in Definition \ref{d3.1} and
set $D_k\equiv S_k-S_{k-1}$ for all $k\in\zz$.

(i) If $q\in (0,\fz]$, then $f\in\b$ if and only if $f\in\sh$ and
\begin{equation}\label{1.2}
    \lf\{\dsum^\fz_{k=-\fz}2^{ksq}\|D_k(f)\|_\lp^q\r\}^{1/q}\equiv J_1<\fz;
\end{equation}
 moreover, $\|f\|_\b$ is equivalent to $\|f\|_{h^p(\cx)}+J_1.$

(ii) If $q\in (n/(n+s), \fz]$, then $f\in\f$ if and only if
$f\in\sh$ and
\begin{equation*} \lf\|\lf\{\dsum^\fz_{k=-\fz}2^{ksq}
|D_k(f)|^q\r\}^{1/q}\r\|_\lp\equiv J_2<\fz;
\end{equation*}
 moreover, $\|f\|_\f$ is equivalent to $\|f\|_{h^p(\cx)}+J_2.$
\end{thm}

The proof of Theorem \ref{t1.2} is given in Section \ref{s4}.

\begin{rem}\label{r1.1}
(i) It is known that when $p\in (1,\fz)$, $h^p(\cx)=\lp$ and
when $p\in (n/(n+1), 1]$, $(\sh\cap L^p_\loc(\cx))\subsetneq\lp;$
see \cite{hmy2}. Thus, when $p\in
(n/(n+1), 1]$, even for the Euclidean space $\rn$, Theorem
\ref{t1.2} is also an improvement of the known classical results;
see \cite[Theorem 2.3.3]{t92}.

(ii) Theorem \ref{t1.2} with $p\in [1,\fz]$ and $\sh$ replaced by
$\lp$ was obtained in \cite[Proposition 5.39]{hmy2}. However,
differently from the Euclidean space (see \cite[Theorem
2.3.3]{t92}), it is unclear if Theorem \ref{t1.2} is still true if
$\sh$ is replaced by $\lp$ when $p\in(n/(n+s), 1)$.

(iii) When $p=\fz$, it is known that $F^0_{p,2}(\cx)=\sbmo(\cx)$;
see \cite[Theorem 6.28]{hmy2}. In this case, if we replace $\sh$ by
$\sbmo(\cx)$ in Theorem \ref{t1.2}, all conclusions of Theorem
\ref{t1.2} are still true. This can be deduced from the
corresponding conclusions described in (ii) of this remark and an
easy argument as in the proof of Theorem \ref{t1.2}.

(iv) Usually, it makes no sense to write the conclusions of
Theorem \ref{t1.2} into $\b=(\sh\cap\db)$ and $\f=(\sh\cap\df)$,
since homogeneous and inhomogeneous spaces are defined via different
kinds of spaces of distributions, which was pointed out by Professor
Hans Triebel to the first author.
\end{rem}

In Section \ref{s5} of this paper, via the discrete Calder\'on
reproducing formulae in \cite{hmy2}, we study the locally
integrability of elements in Besov and
Triebel-Lizorkin spaces; see Propositions \ref{p5.2} and
\ref{p5.3} below. Such results on $\rn$ were obtained by Sickel and
Triebel \cite{st}. However, the method used in \cite{st} is not
valid for RD-spaces $\cx$, since there exists none
counterpart on $\cx$ of the embedding theorems for different metrics
as in Triebel \cite[p.\,129]{t83} on $\rn$; see the proof of
\cite[Theorem 3.3.2]{st}.

The results in this paper apply in a wide range of settings,
for instance, to Ahlfors $n$-regular metric measure spaces
(see \cite{hei}), $d$-spaces (see \cite{t06}),
Lie groups of polynomial volume growth (see
\cite{vsc, nsw, a}), (compact) Carnot-Carath\'eodory
(also called sub-Riemannian) manifolds
with doubling measures (see \cite{nsw,hk}) and
to boundaries of certain unbounded model domains of polynomial type in $\cc^N$
appearing in the work of Nagel and Stein (see \cite{ns06,nsw}).

\begin{rem}
We point out that in the original definition of spaces of homogeneous type
in \cite{cw1,cw2}, $d$ is only assumed to be a quasi-metric instead of a metric.
As pointed out by \cite[Remark 1.4]{hmy2}, all the results in this paper
are still true for quasi-metrics having some regularity.
Moreover, Mac\'ias and Segovia \cite{ms79a} proved that
if $d$ is a quasi-metric, then there exists
a quasi-metric $\wz d$, which is equivalent
to $d$ and has some regularity.
Thus, all the results of this paper
are still true if $d$ is only known to be a quasi-metric
since they are invariant under equivalent
quasi-metrics.
\end{rem}

Finally, we state some conventions. Throughout the paper,
we denote by $C$ a positive
constant which is independent
of the main parameters, but it may vary from line to line.
Constants with subscripts, such as $C_0$, do not change
in different occurrences. The symbol $A\ls B$ or $B\gs A$
means that $A\le CB$. If $A\ls B$ and $B\ls A$, we then
write $A\sim B$. For any
$p\in [1,\fz]$, we denote by $p'$ its conjugate index,
namely, $1/p+1/p'=1$.  If $E$ is a subset of $\cx$,
we denote by $\chi_E$ the characteristic
function of $E$ and define
$\diam E\equiv\sup_{x,\,y\in E}d(x,y).$
We also set $\nn\equiv\{1,\, 2,\, \cdots\}$ and $\zz_+\equiv\nn\cup\{0\}$.
For any $a,\, b\in\rr$, we denote $\min\{a,\, b\}$, $\max\{a,\, b\}$,
and $\max\{a,\, 0\}$
by $a\wg b$, $a\vee b$ and $a_+$, respectively.
For any measurable set $E$ and locally integrable function $f$,
we set $m_E(f)\equiv\frac1{\mu(E)}\int_Ef(x)\,d\mu(x)$.

\section{Characterizations of RD-spaces\label{s2}}

\hskip\parindent In this section, we establish several equivalent characterizations of RD-spaces
in Proposition \ref{p2.1} below, which should be useful in applications.

\begin{prop}\label{p2.1}
The following statements are equivalent.

(i) The triple $(\cx,\,d,\,\mu)$ is an RD-space.

(ii) The triple $(\cx,\,d,\,\mu)$ is a space of homogeneous type and
there exists a constant $a_0>1$
such that for all $x\in\cx$ and $0<r<\diam(\cx)/a_0$,
\begin{equation}\label{2.1}
 B(x, a_0r)\setminus  B(x, r)\ne\emptyset\ (geometrical\ property).
\end{equation}

(iii) The triple $(\cx, d, \mu)$ is an $(\kz, n)$-space for some
$0<\kz\le n$, which means that there exist constants $C_1\ge 1$ and $0<C_2\le1$
such that for all $x\in\cx$,
$0<r<2\diam(\cx) $ and $1\le\lz<2\diam(\cx)/r$,
\begin{equation}\label{2.2}
\mu(B(x, \lz r))\le C_1\lz^n\mu(B(x, r))
\end{equation}
and
\begin{equation}\label{2.3}
\mu(B(x, \lz r))\ge C_2\lz^\kz\mu(B(x, r)).
\end{equation}
\end{prop}

Recall that a metric space satisfying \eqref{2.1} is usually called uniformly perfect;
see, for example, \cite[p.\,88]{hei}. Moreover, parts of the proof of
Proposition \ref{p2.1} can be found in \cite{hei},
for example, the proof of ``(ii) implies (iii)''
in Proposition \ref{p2.1} is just \cite[p.\,31,\, (4.6) and
p.\,102,\,Exercise 13.1]{hei}.
(We thank the referees to point out these facts to us.)
However, for the convenience of readers and its importance in applications,
we would like to give a detailed proof of Proposition \ref{p2.1}.
Indeed,  we conclude
Proposition \ref{p2.1} from Lemmas \ref{l2.1}
through \ref{l2.4} below.

\begin{lem} \label{l2.1}
The following statements are equivalent.

(D1) There exist constants $C_1 \ge 1$, $\ez_0>1$ and $n>0$ such that
\eqref{2.2} holds for all $x\in\cx$,
$0<r<\ez_0\diam(\cx)$ and $1\le\lz<\ez_0\diam(\cx)/r$.

(D2) There exist constants $C_1\ge 1$, $\ez_0>1$ and $n>0$ such that
\eqref{2.2} holds for all $x\in\cx$,
$0<r<\ez_0\diam(\cx)$ and $\lz\ge1$.

(D3) There exist constants $C_1\ge 1$ and $n>0$ such that
\eqref{2.2} holds for all $x\in\cx$, $r>0$ and $\lz\ge1$.

(D4) $(\cx, \,d,\mu)$ is a space of homogeneous type.
\end{lem}

\begin{proof}
If $\diam(\cx)=\fz$, then  (D1), (D2) and (D3) of Lemma \ref{l2.1}
are the same; otherwise, it is easy to see that
(D1) implies (D2), (D2) implies (D3) and (D3) implies (D1). Moreover,  (D3)
implies (D4) with the doubling constant $C_0\equiv C_12^n$ and (D4) implies
(D3) with $n\equiv\log_2C_0$ and $C_1\equiv C_0$, which completes the proof of Lemma \ref{l2.1}.
\end{proof}

\begin{lem}\label{l2.2}
The following statements are equivalent.

(RD1) There exist constants $\ez_1>0$, $0<C_2\le 1$  and $\kz>0$ such that
\eqref{2.3} holds for all $x\in\cx$,
$0<r<\ez_1\diam(\cx)$ and $1\le\lz<\ez_1\diam(\cx)/r$.

(RD2) For any $\ez_1>0$, there exist constants $0<C_2\le 1$ and
$\kz>0$ such that \eqref{2.3} holds for all $x\in\cx$,
$0<r<\ez_1\diam(\cx)$ and $1\le\lz<\ez_1\diam(\cx)/r$.

(RD3) There exist constants $\ez_1>0$, $\wz C_2>1$ and
$a_1\in(1,\,\fz)\cap[\ez_1,\,\fz)$ such that
 for all $x\in\cx$ and $0<r<\ez_1 \diam(\cx)/a_1$,
\begin{equation}\label{2.4}
\mu(B(x,\,a_1r))\ge \wz C_2\mu(B(x,\,r)).
\end{equation}

(RD4) For any $\ez_1>0$, there exist constants $\wz C_2>1$  and
$a_1\in(1,\,\fz)\cap[\ez_1,\,\fz)$ such that \eqref{2.4} holds for
all $x\in\cx$ and $0<r< \ez_1 \diam(\cx)/a_1$.
\end{lem}

\begin{proof}
 Obviously, (RD2) implies (RD1) and (RD4) implies (RD3).
Moreover, (RD1) implies (RD3) by choosing $a_1\in(1,\,\fz)\cap[\ez_1,\,\fz)$ such that
$\wz C_2\equiv C_2a_1^\kz>1$.

Now we prove (RD3) implies (RD1). Let $\ez_1$ be as in (RD3),
$\kz=\log_{a_1}\wz C_2$ and $C_2=\wz C_2^{-1}$. Assume that
$a_1^\ell\le\lz<a_1^{\ell+1}$ for some $\ell\ge 0$.  Then for all
$0<r<\ez_1\diam(\cx)$ and $1\le\lz<\ez_1\diam(\cx)/r$, we have
$$\mu(B(x,\,\lz r))\ge \mu(B(x,\, a_1^{\ell}r))\ge
\wz C_2^{\ell}\mu(B(x,\,r))\ge \wz
C_2^{-1}\lz^{\kz}\mu(B(x,\,r)).$$

By  the same proof as above, we also have that (RD4) implies (RD2).

Now we prove (RD1) implies (RD2). In fact, if $0<\ez_2\le\ez_1$, then
(RD2) holds for $\ez_2$. If $\ez_2>\ez_1$,
then since (RD1) holds for $\ez_1$, to prove that (RD2) also holds for $\ez_2$,
we still need to prove that if  $\ez_1\diam(\cx)\le r<\ez_2\diam(\cx)$ and
$1\le\lz<\ez_2\diam(\cx)/r$, or if $0<r<\ez_1\diam(\cx)$ and
$\ez_1\diam(\cx)/r\le\lz<\ez_2\diam(\cx)/r$, then $$\mu(B(x,\,\lz
r))\ge C_2\lz^\kz\mu(B(x,\,r)).$$ In fact, if
$\ez_1\diam(\cx)\le r<\ez_2\diam(\cx)$ and $1\le\lz<\ez_2\diam(\cx)/r$,
then  $1\le\lz<\ez_2/\ez_1$, and hence
$$\mu(B(x,\,\lz r))\ge [\ez_2/\ez_1]^{-\kz}\lz^\kz\mu(B(x,\,r)).$$
If  $0<r<\ez_1\diam(\cx)$
and $\ez_1\diam(\cx)/r\le\lz<\ez_2(\diam \cx)/r$,
choosing $$\wz\lz\equiv \max\{\ez_1\diam(\cx)/(2r),\,1\},$$
then $\wz\lz/\lz\ge\ez_1/(2\ez_2)$.
In fact, if $r\ge \ez_1\diam(\cx)/2$, then $\wz\lz=1$
and hence $\wz\lz/\lz\ge \ez_1/\ez_2$, and
if $r< \ez_1\diam(\cx)/2$, then $\wz\lz=\ez_1\diam(\cx)/(2r)$
and hence $\wz\lz/\lz\ge\ez_1/(2\ez_2)$.
Since $B(x,\,\wz\lz r)\subset B(x,\,\lz r)$ and $\wz\lz r<\ez_1\diam(\cx)$,
we have
\begin{eqnarray*}
\mu(B(x,\,\lz r))&&\ge\mu(B(x,\,\wz\lz r))\ge C_2 (\wz\lz)^\kz \mu(B(x,\,r))\\
&&=C_2(\wz\lz/\lz)^\kz
\lz^\kz\mu(B(x,\,r))\ge C_2[\ez_1/(2\ez_2)]^{\kz}\lz^\kz\mu(B(x,\,r)),
\end{eqnarray*}
which is desired.

By  the same proof as above, (RD3) also implies (RD4),
which completes the proof of Lemma \ref{l2.2}.
\end{proof}

\begin{lem}\label{l2.3}
The following geometric properties are equivalent.

(G1) There exist constants $\ez_0>0$ and
$a_0\in(1,\,\fz)\cap[\ez_0,\,\fz)$ such that \eqref{2.1} holds for
all $x\in\cx$ and $0<r<\ez_0\diam(\cx)/a_0$.

(G2) For any $\ez_0>0$, there exist constants
$a_0\in(1,\,\fz)\cap[\ez_0,\,\fz)$ such that \eqref{2.1} holds for
all $x\in\cx$ and $0<r<\ez_0\diam(\cx)/a_0$.
\end{lem}

\begin{proof}
(G2) obviously implies (G1). On the other hand,  if $\ez_1\le\ez_0$,
since (G1) holds for $\ez_0$, then (G2) holds for $\ez_1$ with the same $a_0$.
If  $\ez_1>\ez_0$,  taking $a_1\equiv a_0\ez_1/\ez_0$, we know that
(G2) holds for $\ez_1$ and $a_1$, which completes the proof of Lemma \ref{l2.3}.
\end{proof}

\begin{lem}\label{l2.4}
(i) (RD3) implies (G1) with
the same constants.

(ii) If $\ez_0>1$, then
(G1) and (D3) imply (RD3) with $\ez_1\equiv\ez_0/2$
and $a_1\equiv1+2a_0$.
\end{lem}

\begin{proof} (i) obviously holds. To see (ii),
if $0<r<\ez_1\diam(\cx)/a_0$, we then have
$0<2r<\ez_0\diam(\cx)/a_0$. Thus, by (G1), we have
$B(x,\,2a_0r)\setminus B(x,\,2r)\ne\emptyset.$ Choose $y\in
B(x,\,2a_0r)\setminus B(x,\,2r)$. Then $B(y,\, r)\cap
B(x,\,r)=\emptyset$. Notice that $$B(y,\,r)\subset
B(x,\,[1+2a_0]r)\subset  B(y,\,[1+4a_0]r).$$ By (D3), we have
\begin{eqnarray*}
\mu(B(x,\,[1+2a_0]r))
&&\ge\mu(B(x,\,r))+\mu(B(y,\,r))\\
&&\ge\mu(B(x,\,r))+ C_1^{-1} (1+4a_0)^{-n}\mu(
B(y,\,[1+4a_0]r))\\
&&\ge\mu(B(x,\,r))+ C_1^{-1}(1+4a_0)^{-n}\mu( B(x,\,[1+2a_0]r)),
\end{eqnarray*}
which implies that
$\mu(B(x,\,a_1r)) \ge\wz C_2\mu(B(x,\,r))$
with  $a_1\equiv1+2a_0$ and
$$\wz C_2\equiv\lf[1-C_1^{-1} (1+4a_0)^{-n}\r]^{-1}>1.$$
This implies that (RD3) holds and hence finishes the proof of Lemma
\ref{l2.4}.
\end{proof}

\begin{proof}[Proof of Proposition \ref{p2.1}]

Notice that $\cx$ is an RD-space if and only if
$\cx$ satisfies (D4) and (RD3) with $\ez_1\equiv1$,
 $\wz C_2\equiv 2$ and $a_0\equiv a_1$;
$\cx$ is an $(\kz,\,n)$-space
if and only if
$\cx$ satisfies (D1) and (RD1) with $\ez_1\equiv\ez_0\equiv2$;
$\cx$ satisfies (ii) of Proposition \ref{p2.1}
if and only if
$\cx$ satisfies (D4) and (G1) with $\ez_0\equiv1$.
Then Proposition \ref{p2.1} follows from Lemmas \ref{l2.1} through \ref{l2.4},
which completes the proof of Proposition \ref{p2.1}.
\end{proof}

By the well-known fact that connected spaces are uniformly perfect (see, for example, \cite[p.\,88]{hei}),
as a corollary of Porosition \ref{p2.1}, we know that all connected spaces of homogeneous type are RD-spaces.

\begin{rem} \label{r2.1}

(i) The numbers $\kz$ and $n$ appearing in the definition of $(\kz,\,n)$-space here
 measure the ``dimension" of $\cx$ in some sense.

(ii) We point out that the definition of $(\kz,\,n)$-spaces $\cx$ is slightly
different from that of \cite[Definition 1.1]{hmy2}, where
\eqref{2.3} and \eqref{2.2} are assumed to hold only
when $0<r<\diam(\cx)/2$ and $1\le\lz<\diam(\cx)/(2r)$,
from which it is still unclear how to deduce that
$\cx$ is a space of homogeneous type
when $\diam(\cx)<\fz$.

(iii) By Proposition \ref{p2.1}, the definition of RD-spaces
in \cite[Definition 1.1]{hmy2} is equivalent to Definition \ref{d1.1}.
\end{rem}

\section{Properties of spaces of test functions}\label{s3}

\hskip\parindent We prove Theorem \ref{t1.1} in this section.
To this end, we need the
homogeneous and inhomogeneous Calder\'on reproducing formulae
established in Theorems 3.10 and 3.26 of \cite{hmy2}, respectively.

We begin with the following notion of approximations of the identity
on RD-spaces introduced in
 \cite{hmy2}; see  \cite[Theorem 2.6]{hmy2} for its existence.

\begin{defn}\hspace{-0.2cm}{\bf}\label{d3.1}
(I) Let $\ez_1\in(0, 1]$. A sequence $\{S_k\}_{k\in\zz}$ of bounded
linear integral operators on $L^2(\cx)$ is called an
approximation of the identity of order $\ez_1$
with bounded support, if there exist
constants $C_3$, $C_4>0$ such that for all $k\in\zz$ and all $x$,
$x'$, $y$ and $y'\in\cx$, $S_k(x, y)$, the integral kernel of $S_k$
is a measurable function from $\cx\times\cx$ into $\cc$ satisfying
\begin{enumerate}
\item[(i)] $S_k(x, y)=0$ if $d(x, y)>C_42^{-k}$
and $|S_k(x, y)|\le C_3\frac{1}{V_{2^{-k}}(x)+V_{2^{-k}}(y)};$
\vspace{-0.3cm}
\item[(ii)] $|S_k(x, y)-S_k(x', y)|
\le C_3 2^{k \ez_1}
\frac{[d(x, x')]^{\ez_1}}{V_{2^{-k}}(x)+V_{2^{-k}}(y)}$ for $d(x, x')\le\max\{C_4,
1\}2^{1-k}$;
\vspace{-0.3cm}
\item[(iii)] Property (ii) holds with $x$ and $y$ interchanged;
\vspace{-0.3cm}
\item[(iv)] $|[S_k(x, y)-S_k(x, y')]-[S_k(x', y)-S_k(x', y')]|
\le C_3 2^{2k\ez_1}\frac{[d(x, x')]^{\ez_1}[d(y, y')]^{\ez_1}}
{V_{2^{-k}}(x)+V_{2^{-k}}(y)}$ for $d(x, x')\le\max\{C_4,
1\}2^{1-k}$ and $d(y, y')\le\max\{C_4, 1\}2^{1-k}$;
\vspace{-0.3cm}
\item[(v)] $\int_\cx S_k(x, w)\, d\mu(w)=1
=\int_\cx S_k(w, y)\, d\mu(w)$.
\end{enumerate}

(II) A sequence $\{S_k\}_\kzz$ of linear operators
is called an {\it inhomogeneous approximation
to the identity of order $\ez_1$
with bounded support,}
if $S_k$ for $k\in\zz_+$ satisfies (I).
\end{defn}

\begin{lem}\label{l3.1}
Let $\ez\in (0,1)$ and $\{S_k\}_{k\in\zz}$ be an approximation
of the identity of order $1$ with bounded support.
Set $D_k\equiv S_k-S_{k-1}$ for $k\in\zz$. Then there exists
a family $\{\wz D_k\}_{k\in\zz}$ of linear operators
such that for all $f\in\ocg_0^\ez(\bz,\gz)$ with $\bz,\ \gz\in(0,\ez)$,
\begin{equation*}
f=\sum_{k=-\fz}^\fz\wz D_kD_k(f),
\end{equation*}
where the series converges in both the norm of $\ocg_0^\ez(\bz,\gz)$
and the norm of $\lp$ for $p\in (1,\fz)$.
Moreover, for any $\ez'\in (\ez,1)$, there exists a positive
constant $C_{\ez'}$ such that the kernels, denoted by $\{\wz D_k(x,y)\}_{k\in\zz}$,
of the operators $\{\wz D_k\}_{k\in\zz}$ satisfy that for all $x,\ x',\ y\in\cx$,
\begin{enumerate}
    \item[(i)] $|\wz D_k(x,y)|\le C_{\ez'}\frac 1{V_{2^{-k}}(x)+V_{2^{-k}}(y)+V(x,y)}
    \frac {2^{-k\ezp}}{[2^{-k}+d(x,y)]^\ezp};$
    \item[(ii)] $|\wz D_k(x,y)-\wz D_k(x',y)|\le C_{\ez'}\lf[\frac{d(x,x')}
    {2^{-k}+d(x,y)}\r]^{\ezp}\frac 1{V_{2^{-k}}(x)+V_{2^{-k}}(y)+V(x,y)}
    \frac {2^{-k\ezp}}{[2^{-k}+d(x,y)]^\ezp}$ for \newline
    $d(x,x')\le[2^{-k}+d(x,y)]/2;$
    \item[(iii)] $\int_\cx \wz D_k(x,w)\,d\mu(w)=0
    =\int_\cx  \wz D_k(w,y)\,d\mu(w).$
\end{enumerate}
\end{lem}

The following is the inhomogeneous version of Lemma \ref{l3.1}.

\begin{lem}\label{l3.2}
Let $\ez\in (0,1)$ and $\{S_k\}_{k\in\zz_+}$ be an inhomogeneous approximation
of the identity of order $1$ with bounded support.
Set $D_k\equiv S_k-S_{k-1}$ for $k\in\nn$ and $D_0\equiv S_0$. Then there exists
a family $\{\wz D_k\}_{k\in\zz_+}$ of linear operators
such that for all $f\in\cg_0^\ez(\bz,\gz)$ with $\bz,\ \gz\in(0,\ez)$,
\begin{equation*} \label{4.x}
    f=\sum_{k=0}^\fz\wz D_kD_k(f),
\end{equation*}
where the series converges in both the norm of $\cg_0^\ez(\bz,\gz)$
and the norm of $\lp$ for $p\in (1,\fz)$. Moreover, for any $\ez'\in
(\ez,1)$, there exists a positive constant $C_{\ez'}$ such that the
kernels, denoted by $\{\wz D_k(x,y)\}_{k\in\zz_+}$, of the operators
$\{\wz D_k\}_{k\in\zz_+}$ satisfy (i) and (ii) of Lemma \ref{l3.1}
and
\begin{enumerate}
    \item[(iii)'] $\int_\cx  \wz D_0(x,w)\,d\mu(w)=1
    =\int_\cx  \wz D_0(w,y)\,d\mu(w)$ and for all $k\in\nn$,
    $\int_\cx  \wz D_k(x,w)\,d\mu(w)=0
    =\int_\cx  \wz D_k(w,y)\,d\mu(w)$.
\end{enumerate}
\end{lem}

\begin{proof}[Proof of Theorem \ref{t1.1}]

Without loss of generality, we may assume that
$\ez<\wez$. Since $\cg(\wez,\wez)\subset\cg(\ez,\ez)$ and
$\ocg(\wez,\wez)\subset\ocg(\ez,\ez)$, we then have
$\cg^\wez_0(\bz,\gz) \subset\cg^\ez_0(\bz,\gz)$ and
$\ocg^\wez_0(\bz,\gz) \subset\ocg^\ez_0(\bz,\gz)$.

We now show that
    $\cg^\ez_0(\bz,\gz)
\subset\cg^\wez_0(\bz,\gz).$  To this end, it suffices to prove that
\begin{equation}\label{3.1}
    \cg(\ez,\ez)
\subset\cg^\wez_0(\bz,\gz).
\end{equation}
To do so, choose a radial function
$\phi\in C^1(\rr)$ such that $\supp\phi\subset (-2,2)$ and $\phi(x)=1$
if $x\in(-1,1)$. Fix any $\psi\in\cg(\ez,\ez)$. For all $N\in\nn$ and all $x\in\cx$,
let $\psi_N(x)\equiv\psi(x)\phi(\frac {d(x,x_1)}N)$. We first show that
as $N\to\fz$,
\begin{equation}\label{3.2}
    \|\psi-\psi_N\|_\cgbz\to 0.
\end{equation}
In fact, for all $x\in\cx$, we have
\begin{eqnarray}\label{3.3}
&&|\psi(x)-\psi_N(x)|\\
&&\qquad=\lf|\psi(x)\lf[1-\phi\lf(\dfrac {d(x,x_1)}N\r)\r]\r|\noz\\
\noz &&\qquad\le \|\psi\|_{\cg(\ez,\ez)}\dfrac 1{V_1(x_1)+V(x_1,x)}
\lf[\dfrac 1{1+d(x,x_1)}\r]^\ez\chi_{_{\{x\in\cx:\ d(x,x_1)>N\}}}(x)\\
\noz &&\qquad\ls\dfrac 1{N^{\ez-\gz}}\|\psi\|_{\cg(\ez,\ez)}\dfrac 1{V_1(x_1)+V(x_1,x)}
\lf[\dfrac 1{1+d(x,x_1)}\r]^\gz.
\end{eqnarray}
Notice that if $d(x,y)\le\frac 12[1+d(x,x_1)]$ and $d(x,x_1)\le 2N$,
then $V_1(x_1)+V(x_1,y)\gs V_1(x_1)+V(x_1,x)$, $1+d(x_1,y)
\gs 1+d(x_1,x)$ and $d(y,x_1)\le 3N+1/2$, which imply that
for all $x,\,y\in\cx$ with $d(x,\,y)\le[1+d(y,\,x_1)]/2$,
\begin{eqnarray}\label{3.4}
&&|[\psi(x)-\psi_N(x)]-[\psi(y)-\psi_N(y)]|\\
\noz&&\hs=\lf|[\psi(x)-\psi(y)]\lf[1-\phi\lf(\dfrac {d(x,x_1)}N\r)\r]\r.\\
\noz&&\hs\hs\lf.+\psi(y)\lf[\phi\lf(\dfrac {d(y,x_1)}N\r)-\phi\lf(\dfrac {d(x,x_1)}N\r)\r]\r|\\
\noz&&\hs\ls\|\psi\|_{\cg(\ez,\ez)}\lf\{\frac 1{N^{\ez-\bz}}
+\dfrac 1{N^{(1-\bz)\wg(\ez-\gz)}}\r\}\\
\noz&&\hs\hs\times\lf[\dfrac {d(x,y)}{1+d(x,x_1)}\r]^\bz\dfrac 1{V_1(x_1)+V(x_1,x)}
\lf[\dfrac 1{1+d(x,x_1)}\r]^\gz.
\end{eqnarray}
Combining the estimates \eqref{3.3} with \eqref{3.4} yields
\begin{equation*}
    \|\psi-\psi_N\|_\cgbz\ls\|\psi\|_{\cg(\ez,\ez)}\lf\{\frac 1{N^{\ez-\bz}}
+\dfrac 1{N^{(1-\bz)\wg(\ez-\gz)}}\r\}\to 0,
\end{equation*}
as $N\to\fz$. It is easy to see that $\psi_N\in\cg(\ez,\ez')$ for
any $\ez'>0$ by noticing that $\psi_N$ has bounded support. In
particular, $\psi_N\in\cg^\ez_0(\bz,\gz)$. With the notation same as
in Lemma \ref{l3.2}, by Lemma \ref{l3.2}, we have
\begin{equation*}
    \psi_N=\sum^\fz_{k=0}\wdk D_k(\psi_N)
\end{equation*}
in $\cg^\ez_0(\bz,\gz)$, which means that as $L\to\fz$,
\begin{equation*}
    \lf\|\psi_N-\sum_{k=0}^L\wz D_kD_k(\psi_N)\r\|_{\cg(\bz,\gz)}\to 0,
\end{equation*}
where $\wdk(\cdot,y)\in\cg(y, 2^{-k},\ez',\ez')$ for any $\ez'\in
(\wez, 1)$. To finish the proof of \eqref{3.1}, it suffices to
show that
\begin{equation}\label{3.5}
\sum_{k=0}^L\wz D_kD_k(\psi_N)\in\cg(\wez,\wez) \end{equation}
with its norm depending on $L$ and $N$.
To see this, for any $\ez'>0$ and
$k\in\{0,\,1,\,\cdots,\,L\}$,
we have that for all
$x\in\cx$,
\begin{eqnarray}\label{3.6}
    |D_k(\psi_N)(x)|&=&\lf|\dint_{d(x,z)< C_42^{-k+1}}D_k(x,z)\psi_N(z)\,d\mu(z)\r|\\
    \noz&\ls&\dint_\cx\dfrac 1{V_1(x)+V(x,z)}
    \lf[\dfrac 1{1+d(x,z)}\r]^\ezp\\
    \noz&&\times\dfrac 1{V_1(x_1)+V(x_1,z)}
    \lf[\dfrac 1{1+d(x_1,z)}\r]^\ezp\,d\mu(z)\\
    \noz&\ls&\dfrac 1{V_1(x_1)+V(x_1,x)}
    \lf[\dfrac 1{1+d(x_1,x)}\r]^\ezp,
\end{eqnarray}
and that for all $x,\ y\in\cx$ with $d(x,y)\le(C_4\vee1)2^{1-k}$,
\begin{eqnarray}\label{3.7}
    &&|D_k(\psi_N)(x)-D_k(\psi_N)(y)|\\
    \noz&&\hs=\lf|\dint_{d(x,z)< (C_4\vee1)2^{2-k}}[D_k(x,z)-D_k(y,z)]\psi_N(z)\,d\mu(z)\r|\\
    \noz&&\hs\ls d(x,y)\dint_\cx\dfrac 1{V_1(x)+V(x,z)}
    \lf[\dfrac 1{1+d(x,z)}\r]^{1+\ezp}\\
    \noz&&\hs\hs\times\dfrac 1{V_1(x_1)+V(x_1,z)}
    \lf[\dfrac 1{1+d(x_1,z)}\r]^{1+\ezp}\,d\mu(z)\\
    \noz&&\hs\ls\dfrac {d(x,y)}{1+d(x_1,x)}\dfrac 1{V_1(x_1)+V(x_1,x)}
    \lf[\dfrac 1{1+d(x_1,x)}\r]^\ezp,
\end{eqnarray}
where the implicit constants depend on $L$ and $N$.
The estimates \eqref{3.6} and \eqref{3.7} further imply that if
$d(x,y)\le [1+d(x,x_1)]/2$, then
\begin{eqnarray}\label{3.8}
    &&|D_k(\psi_N)(x)-D_k(\psi_N)(y)|\\
    \noz&&\hs\ls\dfrac {d(x,y)}{1+d(x_1,x)}\dfrac 1{V_1(x_1)+V(x_1,x)}
    \lf[\dfrac 1{1+d(x_1,x)}\r]^\ezp
\end{eqnarray}
with the implicit constant depending on $L$ and $N$.

Observe that for $k\in\{0, 1, \cdots, L\}$, $\wdk(\cdot,z)\in\cg(z,
1,\wez,\wez)$ with its norm depending on $L$. By \eqref{3.6} with
$\ezp=\wez$, we obtain that for all $x\in\cx$,
\begin{eqnarray}\label{3.9}
|\wdk D_k(\psi_N)(x)|\ls&&\dint_\cx\dfrac 1{V_1(x)+V(x,z)}
    \lf[\dfrac 1{1+d(x,z)}\r]^\wez\\
    \noz&&\times\dfrac 1{V_1(x_1)+V(x_1,z)}
    \lf[\dfrac 1{1+d(x_1,z)}\r]^\wez\,d\mu(z)\\
    \noz\ls&&\dfrac 1{V_1(x_1)+V(x_1,x)}
    \lf[\dfrac 1{1+d(x_1,x)}\r]^\wez
\end{eqnarray}
with the implicit constant depending on $L$ and $N$,
and that for all $x,\ y\in\cx$ with $d(x,y)\le [1+d(x,x_1)]/4$,
\begin{eqnarray*}
&&|\wdk D_k(\psi_N)(x)-\wdk D_k(\psi_N)(y)|\\
&&\hs=\lf|\dint_\cx[\wdk(x,z)-\wdk(y,z)]
[D_k(\psi_N)(z)-D_k(\psi_N)(x)]\,d\mu(z)\r|\\
&&\hs=\dsum_{i=1}^3\lf|\dint_{W_i}[\wdk(x,z)-\wdk(y,z)]
[D_k(\psi_N)(z)-D_k(\psi_N)(x)]\,d\mu(z)\r|\equiv I_1+I_2+I_3,
\end{eqnarray*}
where $W_1\equiv\{z\in\cx:\ d(x,y)\le [1+d(x,z)]/2\le [1+d(x,x_1)]/4\},$
$$W_2\equiv\{z\in\cx:\ d(x,y)\le [1+d(x,x_1)]/4\le [1+d(x,z)]/2\},$$
and $W_3\equiv\{z\in\cx:\ d(x,y)>[1+d(x,z)]/2\}$. For $I_1$, by
\eqref{3.6} and \eqref{3.8}, we have
\begin{eqnarray*}
I_1&\ls&\dint_\cx\lf[\dfrac {d(x,y)}{1+d(x,z)}\r]^\wez\dfrac 1{V_1(x)+V(x,z)}
    \lf[\dfrac 1{1+d(x,z)}\r]^\wez\\
    &&\times\lf[\dfrac {d(x,z)}{1+d(x_1,x)}\r]^\wez\dfrac 1{V_1(x_1)+V(x_1,x)}
    \lf[\dfrac 1{1+d(x_1,x)}\r]^\wez\,d\mu(z)\\
    &\ls&\lf[\dfrac {d(x,y)}{1+d(x_1,x)}\r]^\wez\dfrac 1{V_1(x_1)+V(x_1,x)}
    \lf[\dfrac 1{1+d(x_1,x)}\r]^\wez,
\end{eqnarray*}
where all implicit constants depend on $L$ and $N$,
and we used the following estimate that for all $x\in\cx$,
\begin{equation}\label{3.10}
    \dint_\cx\dfrac 1{V_1(x)+V(x,z)}
    \lf[\dfrac 1{1+d(x,z)}\r]^\wez\,d\mu(z)\ls 1.
\end{equation}

To estimate $I_2$, by \eqref{3.6} and \eqref{3.10}, we obtain
\begin{eqnarray*}
I_2&\ls&\dint_\cx\lf[\dfrac {d(x,y)}{1+d(x,z)}\r]^\wez\dfrac 1{V_1(x)+V(x,z)}
    \lf[\dfrac 1{1+d(x,z)}\r]^\wez\\
    &&\times\lf\{\dfrac 1{V_1(x_1)+V(x_1,z)}
    \lf[\dfrac 1{1+d(x_1,z)}\r]^{2\wez}+\dfrac 1{V_1(x_1)+V(x_1,x)}
    \lf[\dfrac 1{1+d(x_1,x)}\r]^{2\wez}\r\}\,d\mu(z)\\
    &\ls&\lf[\dfrac {d(x,y)}{1+d(x_1,x)}\r]^\wez\dfrac 1{V_1(x_1)+V(x_1,x)}
    \lf[\dfrac 1{1+d(x_1,x)}\r]^\wez,
\end{eqnarray*}
where the implicit constants depend on $L$ and $N$.

If $z\in W_3$, then $d(x,z)<2d(x,y)\le [1+d(x,x_1)]/2$, which
together with \eqref{3.8} and \eqref{3.10} implies that
\begin{eqnarray*}
I_3&\ls&\dint_\cx\lf[|\wdk(x,z)|+|\wdk(y,z)|\r]
\dfrac {d(x,z)}{1+d(x_1,x)}\dfrac 1{V_1(x_1)+V(x_1,x)}
    \lf[\dfrac 1{1+d(x_1,x)}\r]^\wez\,d\mu(z)\\
    &\ls&\lf[\dfrac {d(x,y)}{1+d(x_1,x)}\r]^\wez\dfrac 1{V_1(x_1)+V(x_1,x)}
    \lf[\dfrac 1{1+d(x_1,x)}\r]^\wez,
\end{eqnarray*}
where the implicit constants depend on $L$ and $N$.

Thus, for all $x,\ y\in\cx$ with $d(x,y)\le [1+d(x,x_1)]/4$,
\begin{eqnarray*}
&&|\wdk D_k(\psi_N)(x)-\wdk D_k(\psi_N)(y)|\\
&&\hs\ls\lf[\dfrac {d(x,y)}{1+d(x_1,x)}\r]^\wez\dfrac 1{V_1(x_1)+V(x_1,x)}
    \lf[\dfrac 1{1+d(x_1,x)}\r]^\wez,
\end{eqnarray*}
which together with \eqref{3.9} further implies that this estimate
also holds for all $x,\ y\in\cx$ with $d(x,y)\le [1+d(x,x_1)]/2$.
Thus, $\{\wdk D_k(\psi_N)\}_{k=0}^L\in\cg(\wez,\wez)$
with their norms depending on $L$ and $N$ and hence
\eqref{3.1} holds. This shows that
$\cg^\ez_0(\bz,\gz)=\cg^\wez_0(\bz,\gz)$.

To finish the proof of Theorem \ref{t1.1}, we still need to
show that
\begin{equation}\label{3.11}
\ocg^\ez_0(\bz,\gz)
\subset\ocg^\wez_0(\bz,\gz).
\end{equation}
To this end, fix any $\psi\in\ocg(\ez,\ez)$ and let $\phi$ be as in
\eqref{3.2}. For any $N\in\nn$, let
$$\psi_N(x)\equiv\psi(x)\phi\lf(\frac {d(x,x_1)}N\r).$$
Let $\{S_k\}_{k\in\zz}$ be as in Lemma \ref{l3.1}. We then define
$$g_N(x)\equiv\psi_N(x)-\lf\{\dint_\cx\psi_N(z)\,d\mu(z)\r\}S_0(x,x_1).$$
Then
\begin{equation}\label{3.12}
    \int_\cx g_N(x)\,d\mu(x)=0,
\end{equation}
and since $\int_\cx\psi(x)\,d\mu(x)=0$, we have
\begin{eqnarray}\label{3.13}
\quad\lf|\dint_\cx\psi_N(z)\,d\mu(z)\r|
&=&\lf|\dint_\cx[\psi(z)-\psi_N(z)]\,d\mu(z)\r|\\
\noz&\ls&\|\psi\|_{\cg(\ez,\ez)}\int_{d(x,x_1)\ge N}\dfrac 1{V_1(x_1)+V(x_1,z)}
\lf[\dfrac 1{1+d(x_1,z)}\r]^\ez\,d\mu(z)\\
\noz&\ls&\dfrac 1{N^{\ezp}}\|\psi\|_{\cg(\ez,\ez)},
\end{eqnarray}
where $\ezp\in (0,\ez)$. From \eqref{3.3}, \eqref{3.4} and
\eqref{3.13}, it follows that
\begin{equation*}
    \|\psi-g_N\|_{\cg(\bz,\gz)}\ls\|\psi\|_{\cg(\ez,\ez)}
    \lf\{\frac 1{N^{\ez-\bz}}+\dfrac 1{N^{(1-\bz)\wg(\ez-\gz)}}
    +\dfrac 1{N^{\ezp}}\r\}\to 0,
\end{equation*}
as $N\to\fz$. By \eqref{3.12} and noticing that $g_N$ has bounded
support, it is easy to see that $g_N\in\ocg(\ez,\ez')$ for any
$\ezp>0$. In particular, $g_N\in\ocg^\ez_0(\bz,\gz)$. With the
notation same as in Lemma \ref{l3.1}, by Lemma \ref{l3.1}, we have
\begin{equation*}g_N=\dsum^\fz_{k=-\fz}\wdk D_k(g_N)\end{equation*}
in $\ocg^\ez_0(\bz,\gz)$, which means that as $L\to\fz$,
\begin{equation*}
    \lf\|g_N-\dsum^L_{k=-L}\wdk D_k(g_N)\r\|_{\cg(\bz,\gz)}\to 0,
\end{equation*}
where $\wdk(\cdot, y)\in\ocg(y, 2^{-k},\ezp,\ezp)$ for any
$\ezp\in(\wez,1)$. An argument similar to \eqref{3.5} gives that
$\sum^L_{k=-L}\wdk D_k(g_N)\in\ocg(\wez,\wez)$ with its norm
depending on $N$ and $L$, which completes the proof of \eqref{3.11}
and hence  the proof of Theorem \ref{t1.1}.
\end{proof}

\section{New characterizations of $B^s_{p,\,q}(\cx)$ and $F^s_{p,\,q}(\cx)$}
\label{s4}

\hskip\parindent This section is devoted to the proof of Theorem \ref{t1.2}.
We first recall
the notions of homogeneous Besov spaces $\db$ and
Triebel-Lizorkin spaces $\df$,
and the inhomogeneous Besov spaces $\b$ and Triebel-Lizorkin spaces $\f$ introduced
in \cite{hmy2}.
We point out that when we mention
the homogeneous Besov spaces $\db$ and
Triebel-Lizorkin spaces $\df$,
we always assume  that $\mu(\cx)=\fz$ since they are
well-defined only when $\mu(\cx)=\fz$.
Denote by $n$ the ``dimension" of $\cx$; see Section \ref{s2}.

\begin{defn}\rm\label{d4.1}
Let $\mu(\cx)=\fz$, $\ez\in (0,1)$ and $\{S_k\}_{k\in\zz}$ be an
approximation of the identity of order $\ez$ with bounded support.
For $k\in\zz$, set $D_k\equiv S_k-S_{k-1}$. Let $0<s<\ez$.

(i) Let $n/(n+\ez)<p\le\fz$ and $0<q\le\fz$. The {\it homogeneous
Besov space $\db$} is defined to be the set of all
$f\in(\ocg^\ez_0(\bz,\gz))'$ for some $\bz,\,\gz$ satisfying
\begin{equation}\label{2.5}
s<\bz<\ez\ \mathrm{and}\ \max\{s-\kz/p,\ n(1/p-1)_+\}<\gz<\ez
\end{equation}
such that
$$\|f\|_\db\equiv\lf\{\sum^\fz_{k=-\fz}2^{ksq}\|D_k(f)\|_\lp^q\r\}^{1/q}<\fz$$
with the usual modifications made when $p=\fz$ or $q=\fz$.

(ii) Let $n/(n+\ez)<p<\fz$ and $n/(n+\ez)<q\le\fz.$ The {\it
homogeneous Triebel-Lizorkin space $\df$} is defined to be the set
of all $f\in(\ocg^\ez_0(\bz,\gz))'$ for some $\bz,\ \gz$
satisfying \eqref{2.5} such that
$$\|f\|_\df\equiv\lf\|\lf\{\sum^\fz_{k=-\fz}
2^{ksq}|D_k(f)|^q\r\}^{1/q}\r\|_\lp<\fz$$ with the usual
modification made when $q=\fz$.
\end{defn}

To define the inhomogeneous Besov and Triebel-Lizorkin spaces,
we need to first recall
the following construction given by Christ in \cite{ch90}, which
provides an analogue of the set of Euclidean dyadic cubes on spaces
of homogeneous type.

\begin{lem}\label{l4.1}
Let $\cx$ be a space of homogeneous type. Then there exists a
collection $\{Q^k_\az\subset\cx:\ k\in\zz,\ \az\in I_k\}$ of open
subsets, where $I_k$ is some index set, and constants $\dz\in (0,1)$
and $C_5,\ C_6>0$ such that
\begin{enumerate}
\vspace{-0.2cm}
\item[(i)] $\mu(\cx\setminus \cup_\az Q^k_\az)=0$ for each fixed $k$
and $Q^k_\az\cap Q^k_\bz=\emptyset$ if $\az\ne\bz$; \vspace{-0.2cm}
\item[(ii)] for any $\az,\ \bz,\ k$ and $l$ with $l\ge k,$ either $Q_\bz^l
\subset Q^k_\az$ or $Q^l_\bz\cap Q^k_\az=\emptyset$; \vspace{-0.2cm}
\item[(iii)] for each $(k,\az)$ and each $l<k$,
there exists a unique $\bz$ such that $Q_\az^k\subset Q^l_\bz$;
\vspace{-0.2cm}
\item[(iv)] $\diam (Q_\az^k)\le C_5\dz^k$;
\vspace{-0.2cm}
\item[(v)] each $Q_\az^k$ contains some ball $B(z^k_\az, C_6\dz^k)$,
where $z^k_\az\in\cx$.
\end{enumerate}
\end{lem}

In fact, we can think of $Q^k_\az$ as being a {\it dyadic cube} with
diameter rough $\dz^k$ and centered at $z^k_\az.$ In what follows,
to simplify our presentation, we always suppose $\dz=1/2$;
see \cite{hmy2} for more details.

In the following, for $k\in\zz$ and $\tau\in\ik$, we denote by
$Q_\tau^{k,\nu},$  $\nu=1,\ 2,\ \cdots, N(k,\tau),$ the set of all
cubes $Q_{\tau'}^{k+j}\subset Q_\tau^k,$ where $Q_\tau^k$ is the
dyadic cube as in Lemma \ref{l4.1} and $j$ is a fixed positive large
integer such that $2^{-j}C_5<1/3.$ Denote by $z_\tau^{k,\nu}$ the
``center" of $\qtn$ as in Lemma \ref{l4.1} and by $y_\tau^{k,\nu}$ a
point in $Q_\tau^{k,\nu}.$

\begin{defn}\rm\label{d4.2} Let $\ez\in (0,1)$ and $\{S_k\}_{k\in\zz_+}$
be an inhomogeneous approximation of the identity of order $\ez$
with bounded support as in Definition \ref{d3.1}.
Set $D_0\equiv S_0$ and $D_k\equiv
S_k-S_{k-1}$ for $k\in\nn$.
Let $\{\qto:\ \tau\in I_0,\
\nu=1,\cdots, N(0,\tau)\}$ with a fixed large $j\in\nn$ be dyadic
cubes as above.
Let $0<s<\ez$.

(i) Let $n/(n+\ez)<p\le\fz$ and $0<q\le\fz$. The {\it Besov space
$\b$} is defined to be the set of all $f\in(\cg^\ez_0(\bz,\gz))'$
for some $\bz,\,\gz$ satisfying
\begin{equation}\label{4.2}
s<\bz<\ez\ \mathrm{and}\ n(1/p-1)_+<\gz<\ez
\end{equation}
such that
\begin{eqnarray*}
\|f\|_\b&\equiv&\lf\{\dsum_{\tau\in I_0}\sum_{\nu=1}^{N(0,\,\tau)}
\mu(Q^{0,\,\nu}_\tau)\lf[m_{Q^{0,\,\nu}_\tau}(|D_0(f)|)\r]^p\r\}^{1/p}\\
&&+\lf\{\dsum^\fz_{k=1}2^{ksq}\|D_k(f)\|_\lp^q\r\}^{1/q}<\fz
\end{eqnarray*}
with the usual modifications made when $p=\fz$ or $q=\fz$.

(ii) Let $n/(n+\ez)<p<\fz$ and $n/(n+\ez)<q\le\fz.$ The {\it
Triebel-Lizorkin space $\f$} is defined to be the set of all
$f\in(\cg^\ez_0(\bz,\gz))'$ for some $\bz,\ \gz$ satisfying
\eqref{4.2} such that
\begin{eqnarray*}
\|f\|_\f&\equiv&\lf\{\dsum_{\tau\in I_0}\sum_{\nu=1}^{N(0,\,\tau)}
\mu(Q^{0,\,\nu}_\tau)\lf[m_{Q^{0,\,\nu}_\tau}(|D_0(f)|)\r]^p\r\}^{1/p}\\
&&+\lf\|\lf\{\dsum^\fz_{k=1}
2^{ksq}|D_k(f)|^q\r\}^{1/q}\r\|_\lp<\fz
\end{eqnarray*}
with the usual modification made when $q=\fz$.
\end{defn}

Recall that the local Hardy spaces
$h^p(\cx)\equiv F^0_{p,2}(\cx)$ when $p\in (n/(n+1),\fz)$; see
\cite{hmy2}.

\begin{defn}\rm\label{d4.3} Let $\ez\in (0,1)$ and $s\in (0,\ez)$.

(i) Assume that $\mu(\cx)=\fz$, and let $\{S_k\}_{k\in\zz}$ be an
approximation of the identity of order $\ez$ with bounded support.
For $k\in\zz$, set $D_k\equiv S_k-S_{k-1}$. Let $n/(n+\ez)<q\le\fz$.
The {\it Triebel-Lizorkin space $\dfy$} is defined to be the set of
all $f\in(\ocg^\ez_0(\bz,\gz))'$ for some $\bz,\ \gz$
satisfying $s<\bz,\ \gz<\ez$ such that
$$\|f\|_\dfy\equiv\dsup_{l\in\zz}\dsup_{\az\in I_l}\lf\{\dfrac 1{\mu(\qal)}
\dint_\qal\dsum^\fz_{k=l}2^{ksq}|D_k(f)(x)|^q\,d\mu(x)\r\}^{1/q}<\fz,$$
where the supremum is taken over all dyadic cubes as in Lemma
\ref{l4.1} and the usual modification is made when $q=\fz$.

(ii) Let $\{S_k\}_{k\in\zz_+}$ be an inhomogeneous approximation of
the identity of order $\ez$ with bounded support. Set $D_k\equiv
S_k-S_{k-1}$ for $k\in\nn$ and $D_0\equiv S_0$.
Let $\{\qto:\ \tau\in I_0,\
\nu=1,\cdots, N(0,\tau)\}$ with a fixed large $j\in\nn$ be dyadic
cubes as above.
Let $0<s<\ez$ and $n/(n+\ez)<q\le\fz$. The
{\it Triebel-Lizorkin space $\fy$} is defined to be the set of all
$f\in\lf(\cg^\ez_0(\bz,\gz)\r)'$ for some $\bz,\ \gz$ satisfying
$s<\bz<\ez$ and $0<\gz<\ez$ such that
\begin{eqnarray*}
\|f\|_\fy&\equiv&\max\lf\{\dsup_{\stz{\tau\in\io}{\nu=1,\,\cdots,N(0,\,\tau)}
}m_{Q^{0,\,\nu}_\tau}(|D_0(f)|),\r.\\
&&\hspace{1.5cm}\lf.\dsup_{l\in\nn}\dsup_{\az\in I_l}\lf[\dfrac
1{\mu(\qal)}
\dint_\qal\dsum^\fz_{k=l}2^{ksq}|D_k(f)(x)|^q\,d\mu(x)\r]^{1/q}\r\}<\fz,
\end{eqnarray*}
where the supremum is taken over all dyadic cubes as in Lemma
\ref{l4.1}, and the usual modification is made when $q=\fz$.
\end{defn}

For a given $\ez\in (0,1)$, it was proved in \cite{hmy2} that the
definitions of the spaces $\db$, $\df$, $\dfy$, $\b$, $\f$ and $\fy$
are independent of the choices of the approximation of the identity
and the distribution space, $(\ocg^\ez_0(\bz,\gz))'$ with
suitable $\bz,\ \gz$, respectively, the inhomogeneous approximation
of the identity and the distribution space,
$(\cg^\ez_0(\bz,\gz))'$ with suitable $\bz,\ \gz$ as in the
above definitions. From Theorem \ref{t1.1},
it is easy to see that the definitions of these spaces are also
independent of $\ez$.
\begin{cor}\label{c4.1}
Let $p\in (n/(n+1), \fz]$ and $s\in (0,1)$.

(i) If $\ez\in (\max\{s, n(1/p-1)_+\},1)$ and $q\in (0, \fz]$, then
the definitions of the spaces $\db$ and $\b$ are independent of the
choices of $\ez$ as above, the approximation of the identity and the
distribution space, $(\ocg^\ez_0(\bz,\gz))'$ with $\bz,\ \gz$
as in \eqref{2.5}, respectively, the inhomogeneous approximation of
the identity and the distribution space,
$\lf(\cg^\ez_0(\bz,\gz)\r)'$ with $\bz,\ \gz$ as in \eqref{4.2}.

(ii) If $q\in (n/(n+1), \fz]$ and $\ez\in (\max\{s, n(1/p-1)_+,
n(1/q-1)_+\},1)$, then the definitions of the spaces $\df$ and $\f$
are independent of the choices of $\ez$ as above, the approximation
of the identity and the distribution space,
$(\ocg^\ez_0(\bz,\gz))'$ with $\bz,\ \gz$ as in \eqref{2.5},
respectively, the inhomogeneous approximation of the identity and
the distribution space, $\lf(\cg^\ez_0(\bz,\gz)\r)'$ with $\bz,\
\gz$ as in \eqref{4.2}.
\end{cor}

\begin{rem}\label{r4.1}
(i) In the definitions of the spaces $\db$, $\df$, $\dfy$, $\b$,
$\f$ and $\fy$ as in Definitions \ref{d4.1}, \ref{d4.2} and
\ref{d4.3}, the approximations of the identity are not necessary to
have bounded support; see \cite{hmy2}. All the conclusions in
Corollary \ref{c4.1} are still true.

(ii) When $s\in (-1,0]$, the spaces $\db$ and $\b$ with $p\in
(n/(n+1), \fz]$ and $q\in (0,\fz]$ and the spaces $\df$ and $\f$
with $p,\ q\in (n/(n+1), \fz]$ are also well defined; see
\cite{hmy2}. Moreover, some conclusions similar
to Corollary \ref{c4.1} are also true for these spaces.

(iii) From now on, when we mention the spaces $\db$, $\df$, $\b$ and
$\f$ with $s\in (0,1)$, we always mean that we choose $\ez$ as in
Corollary \ref{c4.1} and then define these spaces as in Definitions
\ref{d4.1}, \ref{d4.2} and \ref{d4.3}.
\end{rem}

To prove Theorem \ref{t1.2}, we need
the following Calder\'on reproducing formula established in
\cite[Theorem 4.14]{hmy2}.

\begin{lem}\label{l4.2} Let $\ez\in (0,1)$ and
$\{S_k\}_{k\in\zz_+}$ be an inhomogeneous approximation of the
identity of order $1$ with bounded support. Set $D_0\equiv S_0$ and
$D_k\equiv S_k-S_{k-1}$ for $k\in\nn$. Then for any fixed $j$ large
enough, there exists a family $\{\wdk(x,y)\}_{k\in\zz_+}$ of functions
such that for any fixed $\ytn\in\qtn$
with $k\in\nn$, $\tau\in\ik$ and $\nu=1,\cdots,\nkt$ and all
$f\in\lf(\cg^\ez_0(\bz,\gz)\r)'$ with $0<\bz,\ \gz<\ez$ and $x\in\cx$,
\begin{eqnarray*}
f(x)&=&\dsum_{\tau\in
I_0}\dsum^\nz_{\nu=1}\dint_\qto\wdo(x,y)\,d\mu(y)
m_{Q^{0,\,\nu}_\tau}(D_0(f))\\
&&+\dsum^\fz_{k=1}\dsum_{\tau\in\ik}\dsum^\nkt_{\nu=1}
\mu(\qtn)\wdk(x,\ytn)D_k(f)(\ytn),
\end{eqnarray*}
where the series convergence in $\lf(\cg^\ez_0(\bz,\gz)\r)'$.
Moreover, for any $\ez'\in(\ez, 1)$, there exists a positive
constant $C$ depending on $\ez'$ such that the function $\wdk(x,y)$
satisfies (i) and (ii) of Lemma \ref{l3.1} and (iii)' of Lemma
\ref{l3.2} for $k\in\zz_+$.
\end{lem}

\begin{proof}[Proof of Theorem \ref{t1.2}]
For $k\in\zz$, let $D_k\equiv S_k-S_{k-1}.$ We also choose
$\ez\in (s,1)$. We first show (i). If $f\in\sh$ and \eqref{1.2}
holds, by Definition \ref{d4.2}, there exist  $\bz$ and $\gz$ as in
\eqref{4.2} such that $f\in(\cg^\ez_0(\bz,\gz))'$ and
\begin{eqnarray*}
    \|f\|_\sh&\equiv&\lf\{\dsum_{\tau\in I_0}\dsum^\nz_{\nu=1}
\mu(\qto)\lf[m_\qto(|D_0(f)|)\r]^p\r\}^{1/p}\\
\noz&&+\lf\|\lf\{\dsum^\fz_{k=1}|D_k(f)|^2\r\}^{1/2}\r\|_\lp.
\end{eqnarray*}
From this and Definition \ref{d4.2} together with Corollary \ref{c4.1},
it follows that
\begin{eqnarray*}
\|f\|_\b&\sim&\lf\{\dsum_{\tau\in I_0}\dsum^\nz_{\nu=1}
\mu(\qto)\lf[m_\qto(|D_0(f)|)\r]^p\r\}^{1/p}\\
&&+\lf\{\dsum^\fz_{k=1}2^{ksq}\|D_k(f)\|_\lp^q\r\}^{1/q}
\ls\|f\|_\sh+J_1<\fz,
\end{eqnarray*}
where $J_1$ is as in Theorem \ref{t1.2}.
Thus, $f\in\b$.

Conversely, assume that $f\in\b$. By Definition \ref{d4.2} again, we
know that $f\in(\cg^\ez_0(\bz,\gz))'$ for some $\bz,\ \gz$ as in
\eqref{4.2}. Recall that for all $\{a_j\}_j\subset\cc$ and $r\in
(0,1]$,
\begin{equation}\label{4.3}
    \lf(\sum_j|a_j|\r)^r\le\sum_j|a_j|^r.
\end{equation}
If $p/2\le 1$, by \eqref{4.3}, we have
\begin{eqnarray*}
\lf\|\lf\{\dsum^\fz_{k=1}|D_k(f)|^2\r\}^{1/2}\r\|_\lp
&\ls&\lf\{\dsum^\fz_{k=1}\|D_k(f)\|_\lp^p\r\}^{1/p}.
\end{eqnarray*}
From this, \eqref{4.3} when $q/p\le 1$ and the H\"older inequality
when $q/p>1$ together with the assumption that $s>0$, it further
follows that
\begin{eqnarray*}
\lf\|\lf\{\dsum^\fz_{k=1}|D_k(f)|^2\r\}^{1/2}\r\|_\lp
&\ls&\lf\{\dsum^\fz_{k=1}2^{ksq}\|D_k(f)\|_\lp^q\r\}^{1/q}
\ls\|f\|_\b,
\end{eqnarray*}
which together with Definition \ref{d4.2} shows that
$f\in\sh$ and $\|f\|_\sh\ls\|f\|_\b$.

If $p/2>1$, by the H\"older inequality and the assumption that $s>0$,
we have
\begin{eqnarray*}
\lf\|\lf\{\dsum^\fz_{k=1}|D_k(f)|^2\r\}^{1/2}\r\|_\lp
&\ls&\lf\{\dsum^\fz_{k=1}2^{ksp/2}\|D_k(f)\|_\lp^p\r\}^{1/p}.
\end{eqnarray*}
Then an argument similar to the case $p/2\le 1$ also yields
that $f\in\sh$ and $\|f\|_\sh\ls\|f\|_\b$.

On $J_1$, we have
\begin{eqnarray*}
J_1&\ls&\lf\{\dsum^0_{k=-\fz}2^{ksq}\|D_k(f)\|_\lp^q\r\}^{1/q}
+\lf\{\dsum^\fz_{k=1}2^{ksq}\|D_k(f)\|_\lp^q\r\}^{1/q}\\
&\ls& Z_1+\|f\|_\b,
\end{eqnarray*}
where $$Z_1\equiv\lf\{\dsum^0_{k=-\fz}2^{ksq}\|D_k(f)\|_\lp^q\r\}^{1/q}.$$
Using the notation as in Lemma \ref{l4.2}, since
$f\in(\cg^\ez_0(\bz,\gz))'$, by Lemma \ref{l4.2}, for all $z\in\cx$, we have
\begin{eqnarray*}
f(z)&=&\dsum_{\tau'\in\io}\dsum^\nzp_{\nu'=1}\dint_\qtop\wdo(z,y)\,d\mu(y)
m_{Q^{0,\nu'}_{\tau'}}(D_0(f))\\
&&+\dsum^\fz_{k'=1}\dsum_{\tau'\in\ikp}\dsum^\nktp_{\nu'=1}\mu(\qtnp)
\wdkp(z,\ytnp)\dkp(f)(\ytnp)
\end{eqnarray*}
in $(\cg^\ez_0(\bz,\gz))'$. Obviously, $D_k(x,\cdot)\in\cg^\ez_0(\bz,\gz)$.
Thus, we obtain that for all $x\in\cx$,
\begin{eqnarray}\label{4.4}
D_k(f)(x)&=&\dsum_{\tau'\in\io}\dsum^\nzp_{\nu'=1}\dint_\qtop(D_k\wdo)(x,y)\,d\mu(y)
m_{Q^{0,\nu'}_{\tau'}}(D_0(f))\\
\noz&&+\dsum^\fz_{k'=1}\dsum_{\tau'\in\ikp}\dsum^\nktp_{\nu'=1}\mu(\qtnp)
(D_k\wdkp)(x,\ytnp)\dkp(f)(\ytnp)\\
\noz&\equiv&I_1+I_2.
\end{eqnarray}
For any $\ezp\in (\ez,1)$, $k\le 0$ and $\ytop\in\qtop$, by the size conditions of
$\wdo$ and $S_k$, for all $x,\,y\in\cx$, we have
\begin{eqnarray*}
|(D_k\wdo)(x,y)|&=&\lf|\dint_\cx D_k(x,z)\wdo(z,y)\,d\mu(z)\r|\\
\noz&\ls&\dint_{d(x,z)\ge d(x,y)/2}\frac 1{V_{2^{-k}}(x)+V(x,z)}
    \frac {2^{-k\ezp}}{[2^{-k}+d(x,z)]^\ezp}\\
\noz&&\times\frac 1{V_1(z)+V(z,y)}
    \frac 1{[1+d(z,y)]^\ezp}\,d\mu(z)+\dint_{d(z,y)>d(x,y)/2}\cdots\\
\noz&\ls&\frac 1{V_{2^{-k}}(x)+V(x,y)}
    \frac {2^{-k\ezp}}{[2^{-k}+d(x,y)]^\ezp}\\
\noz&&+\frac 1{V_1(x)+V(x,y)}
    \frac 1{[1+d(x,y)]^\ezp}\\
\noz&\ls&\frac 1{V_{2^{-k}}(x)+V(x,\ytop)}
    \frac {2^{-k\ezp}}{[2^{-k}+d(x,\ytop)]^\ezp}\\
\noz&&+\frac 1{V_1(x)+V(x,\ytop)}
    \frac 1{[1+d(x,\ytop)]^\ezp}.
\end{eqnarray*}
From this and Lemma 5.3 of \cite{hmy2} with $n/(n+\ezp)<r\le 1$, it follows
that for all $x\in\cx$,
\begin{eqnarray}\label{4.5}
|I_1|&\ls&\dsum_{\tau'\in I_0}\dsum^\nzp_{\nu'=1}
\mu(\qtop)m_\qtop(|S_0(f)|)\\
\noz&&\times\lf\{\frac 1{V_{2^{-k}}(x)+V(x,\ytop)}
    \frac {2^{-k\ezp}}{[2^{-k}+d(x,\ytop)]^\ezp}\r.\\
\noz&&\lf.+\frac 1{V_1(x)+V(x,\ytop)}
    \frac 1{[1+d(x,\ytop)]^\ezp}\r\}\\
\noz&\ls&2^{kn(1-1/r)}\lf\{M\lf(\dsum_{\tau'\in I_0}\dsum^\nzp_{\nu'=1}
\lf[m_\qtop(|S_0(f)|)\r]^r\chi_\qtop\r)(x)\r\}^{1/r}.
\end{eqnarray}
If we choose $r<p$, by \eqref{4.5} and the
$L^{p/r}(\cx)$-boundedness of $M$, we have
\begin{eqnarray*}
&&\lf\{\dsum^0_{k=-\fz}2^{ksq}\|I_1\|_\lp^q\r\}^{1/q}\\
&&\hs\ls\lf\{\dsum^0_{k=-\fz}2^{k[s+n(1-1/r)]q}\r.\\
&&\hs\hs\times\lf.\lf[\dint_\cx
\lf\{M\lf(\dsum_{\tau'\in I_0}\dsum^\nzp_{\nu'=1}
\lf[m_\qtop(|S_0(f)|)\r]^r\chi_\qtop\r)(x)\r\}^{p/r}\,d\mu(x)\r]^{q/p}\r\}^{1/q}\\
&&\hs\ls\lf(\dsum_{\tau'\in I_0}\dsum^\nzp_{\nu'=1}
\mu(\qtop)\lf[m_\qtop(|S_0(f)|)\r]^p\r)^{1/p}\ls\|f\|_\b.
\end{eqnarray*}

To estimate $I_2$, we first recall that by \cite[Lemma 3.2]{hmy2},
for all $x\in\cx$ and $k\le k'$,
\begin{eqnarray}\label{4.6}
    |(D_k\wdkp)(x,\ytnp)|&\ls& 2^{-(k'-k)\ezp}
    \dfrac 1{V_{2^{-k}}(x)+V_{2^{-k}}(\ytnp)+V(x,\ytnp)}\\
    \noz&&\times\dfrac {2^{-k\ezp}}{[2^{-k}+d(x,\ytnp)]^\ezp}.
\end{eqnarray}
Notice that if $x\in\qtn$, then
\begin{equation}\label{4.7}
V_{2^{-k}}(\ytnp)+V(x,\ytnp)\sim V_{2^{-k}}(\ytnp)+V(\ytn,\ytnp)
\end{equation}
and
\begin{equation}\label{4.8}
    2^{-k}+d(x,\ytnp)\sim 2^{-k}+d(\ytn,\ytnp).
\end{equation}
From \eqref{4.7}, \eqref{4.8}, \eqref{4.3} and Lemma 5.2 of \cite{hmy2},
it follows that when $n/(n+s)<p\le 1$,
\begin{eqnarray*}
&&\lf\{\dsum^0_{k=-\fz}2^{ksq}\|I_2\|_\lp^q\r\}^{1/q}\\
&&\hs\ls\lf(\dsum^0_{k=-\fz}2^{ksq}\lf[\dsum_{\tau\in\ik}
\dsum^\nkt_{\nu=1}\mu(\qtn)\lf\{\dsum^\fz_{k'=1}
\dsum_{\tau'\in\ikp}\dsum^\nktp_{\nu'=1}2^{-(k'-k)\ezp}\mu(\qtnp)\r.\r.\r.\\
&&\hs\hs\lf.\lf.\lf.\times|\dkp(f)(\ytnp)|
\dfrac 1{V_{2^{-k}}(\ytnp)+V(\ytn,\ytnp)}\dfrac {2^{-k\ezp}}
{[2^{-k}+d(\ytn,\ytnp)]^\ezp}\r\}^p\r]^{q/p}\r)^{1/q}\\
&&\hs\ls\lf\{\dsum^0_{k=-\fz}2^{ksq}\lf(\dsum^\fz_{k'=1}
\dsum_{\tau'\in\ikp}\dsum^\nktp_{\nu'=1}2^{-(k'-k)\ezp p}\lf[\mu(\qtnp)
|\dkp(f)(\ytnp)|\r]^p\r.\r.\\
&&\hs\hs\lf.\times\lf[V_{2^{-k}}(\ytnp)\r]^{1-p}\Bigg)^{q/p}\r\}^{1/q}\\
&&\hs\ls\lf\{\dsum^0_{k=-\fz}\lf(\dsum^\fz_{k'=1}
2^{-(k'-k)[\ezp+s-n(1/p-1)]p}2^{k'sp}\r.\r.\\
&&\hs\hs\lf.\lf.\times\dsum_{\tau'\in\ikp}\dsum^\nktp_{\nu'=1}\mu(\qtnp)
|\dkp(f)(\ytnp)|^p\r)^{q/p}\r\}^{1/q}\\
&&\hs\ls\lf\{\dsum^\fz_{k'=1}2^{k'sq}\|\dkp(f)\|_\lp^q\r\}^{1/q}
\ls\|f\|_\b,
\end{eqnarray*}
where in the third-to-last inequality, we used the fact that
\begin{equation*}
V_{2^{-k}}(\ytnp)\ls 2^{(k'-k)n}V_{2^{-k'}}(\ytnp)\sim 2^{(k'-k)n}
\mu(\qtnp),
\end{equation*}
and in the penultimate inequality, we used the arbitrariness of
$\ytnp\in\qtnp$, \eqref{4.3} when $q/p\le 1$, or the H\"older
inequality when $q/p>1$.

From the arbitrariness of $\ytnp\in\qtnp$ again, it is easy to see
that
\begin{eqnarray*}
    &&\dsum^\fz_{k'=1}2^{-(k'-k)(\ezp+s)}
    \dsum_{\tau'\in\ikp}\dsum^\nktp_{\nu'=1}\mu(\qtnp)
    \dfrac 1{V_{2^{-k}}(x)+V(x,\ytnp)}\dfrac {2^{-k\ezp}}
{[2^{-k}+d(x,\ytnp)]^\ezp}\\
&&\hs\ls\dsum^\fz_{k'=1}2^{-(k'-k)(\ezp+s)}\dint_\cx
\dfrac 1{V_{2^{-k}}(x)+V(x,y)}\dfrac {2^{-k\ezp}}
{[2^{-k}+d(x,y)]^\ezp}\,\mu(y)\ls 1.
\end{eqnarray*}
By this estimate, the H\"older inequality, \eqref{3.7} and the
arbitrariness of $\ytnp\in\qtnp$, we obtain that when $p\in
(1,\fz)$,
\begin{eqnarray*}
&&\lf\{\dsum^0_{k=-\fz}2^{ksq}\|I_2\|_\lp^q\r\}^{1/q}\\
&&\hs\ls\lf(\dsum^0_{k=-\fz}\lf\{\dsum^\fz_{k'=1}
\dsum_{\tau'\in\ikp}\dsum^\nktp_{\nu'=1}2^{-(k'-k)(\ezp+s)}2^{k'sp}
\mu(\qtnp)|\dkp(f)(\ytnp)|^p\r.\r.\\
&&\hs\hs\lf.\lf.\times
\dint_\cx\dfrac 1{V_{2^{-k}}(\ytnp)+V(x,\ytnp)}\dfrac {2^{-k\ezp}}
{[2^{-k}+d(x,\ytnp)]^\ezp}\,d\mu(x)\r\}^{q/p}\r)^{1/q}\\
&&\hs\ls\lf\{\dsum^0_{k=-\fz}\lf[\dsum^\fz_{k'=1}2^{-(k'-k)(\ezp+s)}2^{k'sp}
\|\dkp(f)\|_\lp^p\r]^{q/p}\r\}^{1/q}\\\\
&&\hs\ls\lf\{\dsum^\fz_{k'=1}2^{k'sq}\|\dkp(f)\|_\lp^q\r\}^{1/q}
\ls\|f\|_\b,
\end{eqnarray*}
where in the penultimate inequality, we used the arbitrariness of
$\ytnp\in\qtnp$, \eqref{4.3} when $q/p\le 1$, or the H\"older
inequality when $q/p>1$.

Thus, $Z_1\ls\|f\|_\b$ and therefore, $J_1\ls\|f\|_\b$, which
completes the proof of Theorem \ref{t1.2}(i).

To show (ii) of Theorem \ref{t1.2}, if $f\in\sh$ and \eqref{4.2} holds, by an argument
similar to (i), then it is easy to see that $f\in\f$ and
$ \|f\|_\f\ls\|f\|_\sh+J_2,$ where $J_2$ is as in Theorem \ref{t1.2}.

Conversely, if $f\in\f$, since
$$\f\subset B^s_{p,\,\max(p,\,q)}(\cx)$$
(see \cite[Proposition 5.31(iii)]{hmy2}),
we have that $f\in B^s_{p,\,\max(p,\,q)}(\cx)$ and hence,  by (i) of Theorem
\ref{t1.2}, $f\in\sh$ and $\|f\|_\sh\ls\|f\|_\f$.

To estimate $J_2$, we have
\begin{eqnarray*}
J_2&\ls&\lf\|\lf\{\dsum^0_{k=-\fz}2^{ksq}|D_k(f)|^q\r\}^{1/q}\r\|_\lp
+\lf\|\lf\{\dsum^\fz_{k=1}2^{ksq}|D_k(f)|^q\r\}^{1/q}\r\|_\lp\\
&\ls& Z_2+\|f\|_\f,
\end{eqnarray*}
where
$$
Z_2\equiv\lf\|\lf\{\dsum^0_{k=-\fz}2^{ksq}|D_k(f)|^q\r\}^{1/q}\r\|_\lp.$$
Assume that $f\in(\cg^\ez_0(\bz,\gz))'$ for some $\bz,\ \gz$ as in
\eqref{4.2}. Write $D_k(f)$ as in \eqref{4.4}. By \eqref{4.5} with
$n/(n+\ezp)<r\le 1$, $\ezp\in (\ez,1)$ and the
$L^{p/r}(\cx)$-boundedness of $M$ with $r<p$, we obtain
\begin{eqnarray*}
&&\lf\|\lf\{\dsum^0_{k=-\fz}2^{ksq}|I_1|^q\r\}^{1/q}\r\|_\lp\\
&&\hs\ls\lf\|\lf\{\dsum^0_{k=-\fz}2^{k[s+n(1-1/r)]q}\r\}^{1/q}
\lf[M\lf(\dsum_{\tau'\in I_0}\dsum^\nzp_{\nu'=1}
\lf[m_\qtop(|S_0(f)|)\r]^r\chi_\qtop\r)\r]^{1/r}\r\|_\lp\\
&&\hs\ls\lf(\dsum_{\tau'\in I_0}\dsum^\nzp_{\nu'=1}
\mu(\qtop)\lf[m_\qtop(|S_0(f)|)\r]^p\r)^{1/p}\ls\|f\|_\f,
\end{eqnarray*}
where in the penultimate inequality we used the assumption
$p>n/(n+s)$.

To estimate $I_2$, by \eqref{4.6} and Lemma 5.3 of \cite{hmy2} with
$n/(n+\ezp)<r\le 1$, for all $x\in\cx$, we have
\begin{eqnarray*}
|I_2|&\ls&\dsum^\fz_{k'=1}2^{-(k'-k)\ezp}\lf\{M\lf(\dsum_{\tau'\in\ikp}
\dsum^\nktp_{\nu'=1}|\dkp(f)(\ytnp)|^r\chi_\qtnp\r)(x)\r\}^{1/r}\\
&\ls&\dsum^\fz_{k'=1}2^{-(k'-k)\ezp}\lf\{M\lf(|\dkp(f)|^r\r)(x)\r\}^{1/r},
\end{eqnarray*}
where in the last inequality, we used the arbitrariness of
$\ytnp\in\qtnp$. From this, Lemma 3.14 of \cite{hmy2}, \eqref{4.3} when
$q\le 1$,  or the H\"older inequality when $q>1$, it follows that
\begin{eqnarray*}
&&\lf\|\lf\{\dsum^0_{k=-\fz}2^{ksq}|I_2|^q\r\}^{1/q}\r\|_\lp\\
&&\hs\ls\lf\|\lf\{\dsum^0_{k=-\fz}\lf(\dsum^\fz_{k'=1}
2^{-(k'-k)(\ezp+s)}2^{k's}\lf[M\lf(|\dkp(f)|^r\r)\r]^{1/r}\r)^q\r\}^{1/q}\r\|_\lp\\
&&\hs\ls\lf\|\lf\{\dsum^\fz_{k'=1}2^{k'sq}
\lf[M\lf(|\dkp(f)|^r\r)\r]^{q/r}\r\}^{1/q}\r\|_\lp\\
&&\hs\ls\lf\|\lf\{\dsum^\fz_{k'=1}2^{k'sq}
|\dkp(f)|^q\r\}^{1/q}\r\|_\lp\ls\|f\|_\f,
\end{eqnarray*}
where we chose $r<\min\{p,\,q\}$.

Thus, $Z_2\ls\|f\|_\f$ and therefore, $J_2\ls\|f\|_\f$, which
completes the proof of Theorem \ref{t1.2}.
\end{proof}

\section{Local integrability of $\b$, $\f$, $\db$ and $\df$}
\label{s5}

\hskip\parindent We first recall the following (locally) $\lp$-integrability of
elements in Besov spaces and Triebel-Lizorkin spaces, which were
essentially given in the proof of Proposition 4.2 in \cite{my}.
Here we sketch it for the convenience of readers.

\begin{prop}\label{p5.1}
Let $s\in (0,1)$ and $p\in (n/(n+1),\fz]$. Then,

(i) $\db\subset L^p_\loc(\cx)$ for $q\in (0,\fz]$ and $\df\subset
L^p_\loc(\cx)$ for $q\in (n/(n+1),\fz]$;

(ii) $\b\subset\lp$ for $q\in (0,\fz]$ and $\f\subset\lp$ for $q\in
(n/(n+1),\fz]$.
\end{prop}

\begin{proof} It was proved in \cite[(4.4)]{my}
that $\db\subset L^p_\loc(\cx)$ for $s,\,p,\,q$ as in the proposition,
which together with $\df\subset \dot B^s_{p,\,\max(p,q)}(\cx)$ for
$s\in (0,1)$ and $p,\,q\in (n/(n+1),\fz]$ (see \cite[Proposition 5.10(ii)]{hmy2}
and \cite[Proposition 6.9(ii)]{hmy2}) further implies
that $\df\subset L^p_\loc(\cx)$.

To show (ii), by $\f\subset B^s_{p,\,\max(p,\,q)}(\cx)$ (see
\cite[Proposition 5.31(iii)]{hmy2}), it suffices to establish the
conclusion for Besov spaces. Moreover, this was given in the proof
of Proposition 4.2 in \cite{my}, 
which
completes the proof of Proposition \ref{p5.1}.
\end{proof}

As a corollary of Proposition \ref{p5.1} and the H\"older
inequality, we have the following obvious conclusions.

\begin{cor}\label{c5.1}
Let $s\in (0,1)$ and $p\in [1,\fz]$. Then,

(i) $\db\subset L^1_\loc(\cx)$ for $q\in (0,\fz]$ and $\df\subset
L^1_\loc(\cx)$ for $q\in (n/(n+1),\fz]$;

(ii) $\b\subset L^1_\loc(\cx)$ for $q\in (0,\fz]$ and $\f\subset
L^1_\loc(\cx)$ for $q\in (n/(n+1),\fz]$.
\end{cor}

Comparing Corollary \ref{c5.1} with the corresponding conclusions of
Besov and Triebel-Lizorkin spaces on $\rn$ in \cite[Theorem
3.3.2]{st}, the corresponding conclusions for $\cx$ when $p\in
(n/(n+s), 1)$ are missed. To obtain these cases, the method in
\cite{st} strongly depends on the embedding theorems for different
metrics on $\rn$ in \cite[p.\,129]{t83}. However, such embedding
conclusions are not available for $\cx$ due to the fact that for an
RD-space $\cx$, its ``local" dimension may strictly less than its
global dimension such as classes of nilpotent groups; see also
\cite{hmy2}. But, using the inhomogeneous discrete Calder\'on
reproducing formula, Lemma \ref{l4.2}, and some basic properties of
Besov and Triebel-Lizorkin spaces, we can improve Corollary
\ref{c5.1} into the following proposition, which, according to
\cite[Theorem 3.3.2]{st}, is sharp even for Euclidean spaces.

In what follows, for $|s|<1$, let
$$p(s)\equiv\max\{n/(n+1),\ n/(n+1+s)\}.$$

The properties of
Besov and Triebel-Lizorkin spaces on RD-spaces in
the following Lemma \ref{l5.1} can be found
in \cite{hmy2}.

\begin{lem}[\cite{hmy2}]\label{l5.1} Let $|s|<1$.

(i) For $p(s)<p\le\fz$,
$B^s_{p,\, q_0}(\cx)\subset B^s_{p,\, q_1}(\cx)$
when $0<q_0\le q_1\le\fz$,  and
$F^s_{p,\, q_0}(\cx)\subset F^s_{p,\, q_1}(\cx)$
when $p(s)<q_0\le q_1\le\fz$.

(ii) Let $-1<s+\tz<1$ and $\tz>0$. Then for $p(s)<p\le\fz$,
$B^{s+\tz}_{p,\, q_0}(\cx)\subset B^s_{p,\, q_1}(\cx)$
when $0<q_0,\ q_1\le\fz$, and
$F^{s+\tz}_{p,\, q_0}(\cx)\subset F^s_{p,\, q_1}(\cx)$
when $p(s)<q_0,\ q_1\le\fz$.

(iii) If $p(s)<p,\ q\le\fz$, then
$B^s_{p,\,\min(p,\,q)}(\cx)\subset\f\subset B^s_{p,\,\max(p,\,q)}(\cx).$

(iv) $F^0_{p,\,2}(\cx)=\lp$ for $p\in (1,\fz)$, $F^0_{1,\,2}(\cx)=h^1(\cx)$ and
$F^0_{\fz,\,2}(\cx)=\sbmo(\cx)$ with equivalent norms.
\end{lem}

With the aid of Lemma \ref{l5.1}, we further have the following conclusions.

\begin{prop}\label{p5.2} Let $s\in [0,1)$. Then

(i) $\b\subset L^1_\loc(\cx)$ if either $p\in (n/(n+1),\fz]$, $s\in
(n(1/p-1)_+, 1)$, $q\in (0,\fz]$ or $p\in (n/(n+1),1]$,
$s=n(1/p-1)$, $q\in (0,1]$ or $p\in (1,\fz]$, $s=0$, $q\in (0,
\min(p,2)]$;

(ii) $\f\subset L^1_\loc(\cx)$ if either $p\in (n/(n+1),1)$,
$s=n(1/p-1)$, $q\in (n/(n+1),1]$ or $p\in (n/(n+1),1)$, $s\in
(n(1/p-1), 1)$, $q\in (n/(n+1),\fz]$ or $p\in [1,\fz]$, $s\in
(0,1)$, $q\in (n/(n+1),\fz]$ or $p\in [1,\fz]$, $s=0$, $q\in
(n/(n+1), 2]$.
\end{prop}

\begin{rem}\label{r5.1}  Comparing Proposition \ref{p5.2}(ii)
with the sharp result on $\rn$ in \cite[Theorem 3.3.2(i)]{st}, the
conclusion that $\f\subset L^1_\loc(\cx)$ when $p\in (n/(n+1),\,1]$,
$s=n(1/p-1)$ and $q\in (1,\fz]$ is still unknown. But it is easy to
show that this is true if $\cx$ is an Ahlfors $n$-regular metric
measure space, by using the embedding theorem in \cite{y031}.
\end{rem}

\begin{proof}[Proof of Proposition \ref{p5.2}]
To show (i), we consider the following several cases.
{\it Case (i)$_1$} $p\in[1,\fz]$, $s\in (0, 1)$ and $q\in (0,\fz]$.
In this case, by  (ii) and (iii) of Lemma \ref{l5.1}, we have $\b\subset
B^0_{p,\,\min(p,2)}(\cx)\subset F^0_{p,\,2}(\cx)$, which together
with Lemma \ref{l5.1}(iv) and the known facts that $\lp$ for
$p\in(1,\fz)$, $h^1(\cx)$ and $\sbmo(\cx)$ are all subspaces of
$L^1_\loc(\cx)$ further implies that in this case, $\b\subset
L^1_\loc(\cx)$. (This case is also included in Corollary \ref{c5.1}(ii).)

{\it Case (i)$_2$} $p\in (n/(n+1), 1)$, $s>n(1/p-1)>0$ and $q\in
(0,\fz]$ or $p\in (n/(n+1), 1)$, $s=n(1/p-1)$ and $q\in (0,1]$. In
this case, we need to use Lemma \ref{l4.2}. Let all the notation be
as in there. By Lemma \ref{l4.2}, we know that for all $x\in\cx$
\begin{eqnarray*}
f(x)&=&\dsum_{\tau\in I_0}\dsum^\nz_{\nu=1}m_\qto(D_0(f))
\dint_\qto\wdo(x,y)\,d\mu(y)\\
&&+\dsum^\fz_{k=1}\dsum_{\tau\in\ik}\dsum^\nkt_{\nu=1}
\mu(\qtn)\wdk(x,\ytn)D_k(f)(\ytn)
\end{eqnarray*}
holds  in
$(\cg^\ez_0(\bz,\gz))'$ with $\ez\in (s,1)$ and $\bz,\ \gz$ as in
\eqref{4.2}.
To show the conclusions in this case, it suffices to prove that
for any $l_0<0$,
\begin{equation*}
    \int_{B(x_1, 2^{-l_0})}|f(x)|\,d\mu(x)<\fz.
\end{equation*}
To this end, set $E^{k,\nu}_{\tau, 0}
\equiv\{(\tau,\nu):\ \ytn\in B(x_1, 2^{-l_0+2})\}$ and for $l\in\nn$,
\begin{equation*}
E^{k,\nu}_{\tau, l}
\equiv\{(\tau,\nu):\ \ytn\in B(x_1, 2^{-l_0+2+l})\setminus
B(x_1, 2^{-l_0+1+l})\}.
\end{equation*}
Then by the size condition of $\wdk$, for all $x\in\cx$, we have
\begin{eqnarray}\label{5.1}
|f(x)|&\ls&\dsum_{\tau\in I_0}\dsum^\nz_{\nu=1}\mu(\qto)m_\qto(|D_0(f)|)
\frac 1{V_1(x)+V_1(\yto)+V(x,\yto)}\\
\noz&&\times\frac 1{[1+d(x,\yto)]^\ezp}\\
\noz&&+\dsum^\fz_{k=1}\dsum_{\tau\in\ik}\dsum^\nkt_{\nu=1}
\mu(\qtn)\lf|D_k(f)(\ytn)\r|\\
\noz&&\times\frac 1{V_{2^{-k}}(x)+V_{2^{-k}}(\ytn)+V(x,\ytn)}
\frac {2^{-k\ezp}}{[2^{-k}+d(x,\ytn)]^\ezp}\\
\noz&\sim&\dsum^\fz_{l=0}\dsum_{\tau\in I_0}\dsum^\nz_{\nu=1}
\chi_{E^{0,\nu}_{\tau,l}}(\tau,\nu)\mu(\qto)m_\qto(|D_0(f)|)\\
\noz&&\times\frac 1{V_1(x)+V_1(\yto)+V(x,\yto)}\frac 1{[1+d(x,\yto)]^\ezp}\\
\noz&&+\dsum^\fz_{k=1}\dsum^\fz_{l=0}\dsum_{\tau\in\ik}\dsum^\nkt_{\nu=1}
\chi_{E^{k,\nu}_{\tau,l}}(\tau,\nu)\mu(\qtn)\lf|D_k(f)(\ytn)\r|\\
\noz&&\times\frac 1{V_{2^{-k}}(x)+V_{2^{-k}}(\ytn)+V(x,\ytn)}
\frac {2^{-k\ezp}}{[2^{-k}+d(x,\ytn)]^\ezp},
\end{eqnarray}
where $\ez'\in (\ez,1)$.
For any $(\tau,\nu)\in E^{k,\nu}_{\tau,0}$, it is easy to see that
\begin{eqnarray}\label{5.2}
&&\int_{B(x_1, 2^{-l_0})}\frac 1{V_{2^{-k}}(x)+V_{2^{-k}}(\ytn)+V(x,\ytn)}
\frac {2^{-k\ezp}}{[2^{-k}+d(x,\ytn)]^\ezp}\,d\mu(x)\\
\noz&&\hs\ls\int_\cx \frac 1{V_{2^{-k}}(\ytn)+V(x,\ytn)}
\frac {2^{-k\ezp}}{[2^{-k}+d(x,\ytn)]^\ezp}\,d\mu(x)\ls 1
\end{eqnarray}
and that
\begin{equation}\label{5.3}
    \mu(B(x_1, 2^{-l_0}))\le\mu(B(\ytn, 2^{-l_0+3}))\ls 2^{(k-l_0)n}\mu(\qtn).
\end{equation}
For any $(\tau,\nu)\in E^{k,\nu}_{\tau,l}$ with $l\in\nn$
and $x\in B(x_1, 2^{-l_0})$, we have
\begin{equation}\label{5.4}
    d(x,\ytn)\ge d(x_1,\ytn)-d(x,x_1)>d(x_1,\ytn)/2\ge 2^{l-l_0}
\end{equation}
and $B(x_1, 2^{-l_0})\subset B(\ytn, 2^{l+3-l_0})$, which both imply
that
\begin{equation}\label{5.5}
    V(x,\ytn)\sim V(\ytn, x)\gs V(\ytn, x_1)\gs\mu(B(x_1, 2^{l-l_0}))
    \gs 2^{l\kz}\mu(B(x_1, 2^{-l_0}))
\end{equation}
and
\begin{equation}\label{5.6}
    \mu(B(x_1, 2^{-l_0}))\ls 2^{(k+l-l_0)n}\mu(\qtn).
\end{equation}
From the estimates \eqref{5.4} and \eqref{5.5}, it further follows
that for any $(\tau,\nu)\in E^{k,\nu}_{\tau,l}$ with $l\in\nn$,
\begin{eqnarray}\label{5.7}
&&\int_{B(x_1, 2^{-l_0})}\frac 1{V_{2^{-k}}(x)+V_{2^{-k}}(\ytn)+V(x,\ytn)}
\frac {2^{-k\ezp}}{[2^{-k}+d(x,\ytn)]^\ezp}\,d\mu(x)\\
\noz&&\hs\ls\frac {2^{-k\ezp}}{2^{(l-l_0)\ezp}}\frac 1{2^{l\kz}}.
\end{eqnarray}
The estimates \eqref{5.2}, \eqref{5.3}, \eqref{5.6} and
\eqref{5.7} together with \eqref{5.1} and \eqref{4.3} yield that
\begin{eqnarray*}
&&\int_{B(x_1, 2^{-l_0})}|f(x)|\,d\mu(x)\\
&&\hs\ls\dsum^\fz_{l=0}\frac 1{2^{l(\ezp+\kz)}}\dsum_{\tau\in I_0}\dsum^\nz_{\nu=1}
\chi_{E^{0,\nu}_{\tau,l}}(\tau,\nu)\mu(\qto)m_\qto(|D_0(f)|)\\
&&\hs\hs+\dsum^\fz_{k=1}\dsum^\fz_{l=0}\frac 1{2^{l(\ezp+\kz)}}
\dsum_{\tau\in\ik}\dsum^\nkt_{\nu=1}
\chi_{E^{k,\nu}_{\tau,l}}(\tau,\nu)\mu(\qtn)\lf|D_k(f)(\ytn)\r|\\
&&\hs\ls\lf\{\dsum^\fz_{l=0}\frac {2^{ln(1/p-1)}}{2^{l(\ezp+\kz)}}\r\}
\lf[\lf\{\dsum_{\tau\in I_0}\dsum^\nz_{\nu=1}\mu(\qto)
\lf[m_\qto(|D_0(f)|)\r]^p\r\}^{1/p}\r.\\
&&\hs\hs\lf.+\dsum^\fz_{k=1}2^{kn(1/p-1)}
\lf\{\dsum_{\tau\in\ik}\dsum^\nkt_{\nu=1}
\mu(\qtn)\lf|D_k(f)(\ytn)\r|^p\r\}^{1/p}\r]\\
&&\hs\ls\|f\|_\b+\dsum^\fz_{k=1}2^{k[n(1/p-1)-s]}2^{ks}\|D_k(f)\|_\lp,
\end{eqnarray*}
where in the last inequality we used the fact that
$\ezp+\kz>n(1/p-1)$ and the arbitrariness of $\ytn\in\qtn$. Now if
$s>n(1/p-1)$, by the H\"older inequality when $q\in [1,\fz]$ or by
\eqref{4.3} when $q\in (0,1)$, and if $s=n(1/p-1)$ and $q\in
(0,1]$, by \eqref{4.3}, we obtain from the last inequality that
\begin{equation*}
    \int_{B(x_1, 2^{-l_0})}|f(x)|\,d\mu(x)\ls\|f\|_\b<\fz,
\end{equation*}
namely, $f\in L^1_\loc(\cx)$ in this case.

{\it Case (i)$_3$} $p=1$, $s=0$ and $q\in (0,1]$. In this case, by
(i) and (iv) of Lemma \ref{l5.1}, we have $B^0_{1,\,q}(\cx)\subset
B^0_{1,\,1}(\cx)\subset F^0_{1,\,2}(\cx) =L^1(\cx)\subset
L^1_\loc(\cx)$.

{\it Case (i)$_4$} $p\in (1,\fz]$, $s=0$ and $q\in (0, \min(p,2)]$.
In this case, (iii) and (iv) of Lemma \ref{l5.1} and the H\"older
inequality yield that $B^0_{p,\,q}(\cx)\subset F^0_{p,\,2}(\cx)
(=\lp$ when $p\in (1,\fz)$, $=\sbmo(\cx)$ when $p=\fz$) $\subset
L^1_\loc(\cx)$. All these cases complete the proof of (i).

To prove (ii), we also consider the following four cases. {\it Case
(ii)$_1$} $p\in (n/(n+1),1)$, $s=n(1/p-1)>0$ and $q\in
(n/(n+1),1]$. In this case, by (iii) and (i) of Lemma \ref{l5.1}, and the
corresponding conclusion on $\b$, we immediately obtain that
$F^{n(1/p-1)}_{p,\,q}(\cx) \subset
B^{n(1/p-1)}_{p,\,\max{(p,\,q)}}(\cx)\subset
B^{n(1/p-1)}_{p,\,1}(\cx)\subset L^1_\loc(\cx)$.

{\it Case (ii)$_2$} $p\in (n/(n+1),1)$, $s\in (n(1/p-1), 1)$ and
$q\in (n/(n+1),\fz]$. In this case, Lemma \ref{l5.1}(ii) and the
conclusion in Case (ii)$_1$ imply that $\f\subset
F^{n(1/p-1)}_{p,\,1}(\cx)\subset L^1_\loc(\cx)$.

{\it Case (ii)$_3$} $p\in [1,\fz]$, $s\in (0,1)$ and $q\in
(n/(n+1),\fz]$. Lemma \ref{l5.1}(i) yields that $\f\subset
F^0_{p,\,2}(\cx)$, which together with Lemma \ref{l5.1}(iv) and the
known facts that $\lp$ for $p\in(1,\fz)$, $h^1(\cx)$ and
$\sbmo(\cx)$ are all subspaces of $L^1_\loc(\cx)$ further implies
that in this case, $\f\subset L^1_\loc(\cx)$.

{\it Case (ii)$_4$} $p\in [1,\fz]$, $s=0$ and $q\in (n/(n+1), 2]$.
In this case, Lemma \ref{l5.1}(i) implies that $F^0_{p,\,q}(\cx)
\subset F^0_{p,\,2}(\cx)$, which together with Lemma \ref{l5.1}(iv)
and the known facts that $\lp$ for $p\in(1,\fz)$, $h^1(\cx)$ and
$\sbmo(\cx)$ are all subspaces of $L^1_\loc(\cx)$ again further
implies that in this case $F^0_{p,\,q}(\cx)$ is a subspace of
$L^1_\loc(\cx)$. This finishes the proof of Proposition \ref{p5.2}.
\end{proof}

To obtain some conclusions similar to Proposition \ref{p5.2} for the
spaces $\db$ and $\df$, we need to overcome another difficulty,
namely, there exists no counterpart to Lemma \ref{l5.1}(ii) for the
spaces $\db$ and $\df$. However, via Lemma \ref{l5.2}, we can still
improve Corollary \ref{c5.1}(i) in this case into the following
conclusions.

\begin{prop}\label{p5.3} Let $s\in [0,1)$. Then

(i) $\db\subset L^1_\loc(\cx)$ if either $p\in (n/(n+1),\fz]$, $s\in
(n(1/p-1)_+, 1)$, $q\in (0,\fz]$ or $p\in (n/(n+1),1]$,
$s=n(1/p-1)$, $q\in (0,1]$ or $p\in (1,\fz]$, $s=0$, $q\in (0,
\min(p,2)]$;

(ii) $\df\subset L^1_\loc(\cx)$ if either $p\in (n/(n+1),1)$,
$s=n(1/p-1)$, $q\in (n/(n+1),1]$ or $p\in (n/(n+1),1)$, $s\in
(n(1/p-1), 1)$, $q\in (n/(n+1),\fz]$ or $p\in [1,\fz]$, $s\in
(0,1)$, $q\in (n/(n+1),\fz]$ or $p\in [1,\fz]$, $s=0$, $q\in
(n/(n+1), 2]$.
\end{prop}

To prove Proposition \ref{p5.3}, we still need
the following discrete homogeneous
Calder\'on reproducing formula
established in Theorems 4.11 of \cite{hmy2}.

\begin{lem}\label{l5.2}
Let $\ez\in (0,1)$ and $\{S_k\}_{k\in\zz}$ be an approximation of
the identity of order $1$ with bounded support. For $k\in\zz$, set
$D_k\equiv S_k-S_{k-1}$. Then, for any fixed $j\in\nn$ large enough, there
exists a family $\{\wz D_k\}_{k\in\zz}$ of linear operators  such
that for any fixed $\ytn\in\qtn$ with $k\in\zz$, $\tau\in\ik$ and
$\nu=1,\cdots,\nkt$, and all $f\in(\ocg_0^\ez(\bz,\gz))'$ with
$0<\bz,\ \gz<\ez$ and $x\in\cx$,
$$f(x)=\dsum^\fz_{k=-\fz}\dsum_{\tau\in\ik}\dsum^\nkt_{\nu=1}
\mu(\qtn)\wz D_k(x,\ytn)D_k(f)(\ytn),$$ where the series converge in
$(\ocg_0^\ez(\bz,\gz))'$. Moreover, for any $\ez'\in (\ez,1)$,
there exists a positive constant $C_{\ez'}$ such that the kernels,
denoted by $\wz D_k(x,y)$, of the operators $\wz D_k$ satisfy (i),
(ii) and (iii) of Lemma \ref{l3.1}.
\end{lem}

\begin{proof}[Proof of Proposition \ref{p5.3}]

 To establish (i) of Proposition \ref{p5.3}, by Corollary
\ref{c5.1}(i), some properties similar to Lemma \ref{l5.1} (except
(ii); see \cite{hmy2}) for the space $\db$, we only need to consider
the cases when $p\in (n/(n+1),1)$, $s\in (n(1/p-1),1)$ and $q\in
(0,\fz]$ or $p\in (n/(n+1),1)$, $s=n(1/p-1)$ and $q\in (0,1]$. Let
$f\in\db$ with $s$, $p$ and $q$ as above. We show that for any
$l_0\in\zz$,
\begin{equation}\label{5.8}
    \int_{B(x_1,\, 2^{-l_0})}|f(x)|\,d\mu(x)
    \ls 2^{-l_0[s-n(1/p-1)]}\|f\|_{\dot B^s_{p,\,\fz}(\cx)}.
\end{equation}
Let $\{D_k\}_{k\in\zz}$ and other notation be as in Lemma
\ref{l5.2}. Since $f\in\dot B^s_{p,\,\fz}(\cx)$, by Definition
\ref{d4.1}, we know that $f\in(\ocg^\ez_0(\bz,\gz))'$ with
$\bz,\ \gz$ as in \eqref{2.5}, where $\ez\in (s,1)$ and $p\in
(n/(n+\ez),1)$. Thus, for any $f\in\ocg^\ez_0(\bz,\gz)$, since
$\int_\cx g(x)\,d\mu(x)=0$, by Lemma \ref{l5.2}, we have
\begin{eqnarray*}
\laz f, g\raz&=&\dsum_{k=-\fz}^{l_0-1}\dsum_{\tau\in\ik}\dsum^\nkt_{\nu=1}
\mu(\qtn)D_k(f)(\ytn)\lf\laz\wdk(\cdot,\ytn)-\wdk(x_1,\ytn),g\r\raz\\
&&+\dsum_{k=l_0}^\fz\dsum_{\tau\in\ik}\dsum^\nkt_{\nu=1}
\mu(\qtn)D_k(f)(\ytn)\lf\laz\wdk(\cdot,\ytn),g\r\raz.
\end{eqnarray*}
Thus, $f$ is given by the (in $(\ocg^\ez_0(\bz,\gz))'$ with
$\bz,\ \gz$ as in \eqref{2.5}) convergent series of functions
\begin{eqnarray*}
&&\dsum_{k=-\fz}^{l_0-1}\dsum_{\tau\in\ik}\dsum^\nkt_{\nu=1}
\mu(\qtn)D_k(f)(\ytn)\lf[\wdk(x,\ytn)-\wdk(x_1,\ytn)\r]\\
&&\hs\hs+\dsum_{k=l_0}^\fz\dsum_{\tau\in\ik}\dsum^\nkt_{\nu=1}
\mu(\qtn)D_k(f)(\ytn)\wdk(x,\ytn)\equiv Y_1+Y_2.
\end{eqnarray*}
We now prove that \eqref{5.8} holds for $Y_1$ and $Y_2$.

By the regularity of $\wdk$ with $\ezp\in(\ez,1)$, for all $x\in\cx$, we have
\begin{eqnarray*}
|Y_1|&\ls&\dsum_{k=-\fz}^{l_0-1}\dsum_{\tau\in\ik}\dsum^\nkt_{\nu=1}
\mu(\qtn)\lf|D_k(f)(\ytn)\r|\lf[\dfrac {d(x,x_1)}{2^{-k}+d(x_1,\ytn)}\r]^\ezp\\
&&\times\dfrac 1{V_{2^{-k}}(x_1)+V_{2^{-k}}(\ytn)+V(x_1,\ytn)}
\dfrac {2^{-k\ezp}}{[2^{-k}+d(x_1,\ytn)]^\ezp}\\
&\ls&\lf[d(x,x_1)\r]^\ezp\dsum_{k=-\fz}^{l_0-1}2^{k\ezp}
\dsum_{\tau\in\ik}\dsum^\nkt_{\nu=1}
\mu(\qtn)\lf|D_k(f)(\ytn)\r|\\
&&\times\dfrac 1{V_{2^{-k}}(x_1)+V_{2^{-k}}(\ytn)+V(x_1,\ytn)}
\dfrac {2^{-k\ezp}}{[2^{-k}+d(x_1,\ytn)]^\ezp}.
\end{eqnarray*}
Since $p\in (n/(n+1),1)$, for $k\le l_0-1$, it is easy to see that
\begin{eqnarray*}
&&\frac 1{[\mu(\qtn)]^{1-p}}\lf[\dfrac 1{V_{2^{-k}}(x_1)
+V_{2^{-k}}(\ytn)+V(x_1,\ytn)}\r]^p\dfrac {2^{-k\ezp p}}
{[2^{-k}+d(x_1,\ytn)]^{\ezp p}}\\
&&\hs\ls\dfrac 1{V_{2^{-l_0}}(x_1)}\lf[\frac {V_{2^{-k}}(x_1)
+V_{2^{-k}}(\ytn)+V(x_1,\ytn)}{\mu(\qtn)}\r]^{1-p}
\dfrac {2^{-k\ezp p}}
{[2^{-k}+d(x_1,\ytn)]^{\ezp p}}\\
&&\hs\ls\lf\{\begin{array}{l}
\dfrac 1{V_{2^{-l_0}}(x_1)},\quad d(x,\ytn)\le 2^{-k}\\
\dfrac 1{V_{2^{-l_0}}(x_1)}\frac {2^{ln(1-p)}}{2^{l\ezp p}},\quad
2^{-k+l}<d(x,\ytn)\le 2^{-k+l+1},\ l\in\zz_+
        \end{array}\r.\\
&&\hs\ls\dfrac 1{V_{2^{-l_0}}(x_1)}.
\end{eqnarray*}
Using this estimate, we further obtain that for all $x\in\cx$,
\begin{eqnarray*}
|Y_1|&\ls&\lf[d(x,x_1)\r]^\ezp\dsum_{k=-\fz}^{l_0-1}2^{k\ezp}
\lf\{\dsum_{\tau\in\ik}\dsum^\nkt_{\nu=1}
\mu(\qtn)\lf|D_k(f)(\ytn)\r|^p\r\}^{1/p}\\
&\ls&\lf[d(x,x_1)\r]^\ezp\dsum_{k=-\fz}^{l_0-1}2^{k(\ezp-s)}2^{ks}
\|D_k(f)\|_\lp\ls\lf[d(x,x_1)\r]^\ezp\|f\|_\db,
\end{eqnarray*}
where in the penultimate inequality, we used the arbitrariness of
$\ytn\in\qtn$ and in the last inequality, we used the H\"older
inequality when $q\in (1,\fz]$, or \eqref{4.3} when $q\in (0,1]$
together with $\ezp>s$. Thus,
\begin{equation*}
    \int_{B(x_1, 2^{-l_0})}|Y_1|\,d\mu(x)\ls\|f\|_\db<\fz.
\end{equation*}

The estimate for $Y_2$ is similar to the proof of Case (i)$_2$ in
the proof of Proposition \ref{p5.2}(i). Let $E^{k,\nu}_{\tau,l}$
for $l\in\zz_+$ be the same as in the proof of Proposition \ref{p5.2}(i).
By the size condition of $\wdk$, we have
\begin{eqnarray*}
&&\int_{B(x_1, 2^{-l_0})}|Y_2|\,d\mu(x)\\
&&\hs\ls\dsum_{k=l_0}^\fz\dsum^\fz_{l=0}\dsum_{\tau\in\ik}\dsum^\nkt_{\nu=1}
\chi_{E^{k,\nu}_{\tau,l}}(\tau,\nu)\mu(\qtn)\lf|D_k(f)(\ytn)\r|\\
&&\hs\hs\times\int_{B(x_1, 2^{-l_0})}\dfrac 1{V_{2^{-k}}(x)+V_{2^{-k}}(\ytn)+V(x,\ytn)}
\dfrac {2^{-k\ezp}}{[2^{-k}+d(x,\ytn)]^\ezp}\,d\mu(x)\\
&&\hs\ls\dsum_{k=l_0}^\fz\dsum^\fz_{l=0}\dfrac 1{2^{l(\ezp+\kz)}}
\dsum_{\tau\in\ik}\dsum^\nkt_{\nu=1}
\chi_{E^{k,\nu}_{\tau,l}}(\tau,\nu)\mu(\qtn)\lf|D_k(f)(\ytn)\r|\\
&&\hs\ls\dsum_{k=l_0}^\fz\lf[\dsum^\fz_{l=0}\dfrac {2^{ln(1/p-1)}}
{2^{l(\ezp+\kz)}}\r]2^{kn(1/p-1)}\lf\{\dsum_{\tau\in\ik}\dsum^\nkt_{\nu=1}
\mu(\qtn)\lf|D_k(f)(\ytn)\r|^p\r\}^{1/p}\\
&&\hs\ls\dsum_{k=l_0}^\fz2^{k[n(1/p-1)-s]}2^{ks}\|D_k(f)\|_\lp\ls
\|f\|_\db,
\end{eqnarray*}
where we omitted some similar computations to those used in the proof of Case
(i)$_2$ in the proof of Proposition \ref{p5.2}. This finishes the proof
of Proposition \ref{p5.3}(i).

To finish the proof of (ii) of this proposition, we only need to
point out that if $p\in (n/(n+1),1)$, $s\in (n(1/p-1), 1)$ and $q\in
(n/(n+1), \fz]$, then $\df\subset\dot
B^s_{p,\,\max(p,q)}(\cx)\subset L^1_\loc(\cx)$. The other cases are
similar to the proof of Proposition \ref{p5.2}. We omit the details,
which completes the proof of Proposition \ref{p5.3}.
\end{proof}

\begin{rem}\label{r5.2} Similarly to Remark \ref{r5.1},
if $\cx$ is an Ahlfors $n$-regular metric
measure space, $p\in (n/(n+1),1]$, $s=n(1/p-1)$
and $q\in (1,\fz]$, then $\df\subset L^1_\loc(\cx)$.
We omit the details.
\end{rem}

{\bf Acknowledgements.}  The authors sincerely wish to express their
deeply thanks to the referees for their very carefully reading and
also for their suggestive remarks which improve the presentation of this article.

\bigskip

\noindent {\sc Dachun Yang and Yuan Zhou}

\noindent School of Mathematical Sciences, Beijing Normal
University, Laboratory of Mathematics and Complex Systems, Ministry
of Education, Beijing 100875, People's Republic of China

\medskip

\noindent{\it E-mails:} \texttt{dcyang@bnu.edu.cn} and
\texttt{yuanzhou@mail.bnu.edu.cn}

\end{document}